\documentclass[reqno]{amsart}
\usepackage[normalem]{ulem}
\usepackage{kbordermatrix}
\usepackage{hyperref} 
\usepackage{color}
\usepackage{latexsym}
\usepackage{graphicx}
\usepackage{mathrsfs}
\usepackage{amssymb}
\usepackage{amsfonts}
\usepackage{amssymb,amsmath,amsthm,tikz}
\usepackage{enumerate}

\hypersetup{colorlinks = true, urlcolor = black}
\DeclareMathOperator{\Vol}{Vol}

\DeclareMathOperator{\Erf}{Erf}
\DeclareMathOperator{\eq}{eq}

\DeclareMathOperator{\supp}{supp}

\renewcommand{\Re}{\text{Re}}
\renewcommand{\Im}{\text{Im}}

\newcommand{\Szego}{Szeg\"o }
\newcommand{\szego}{Szeg\"o }
\newcommand{\Kahler}{K\"ahler }
\newcommand{\kahler}{K\"ahler }
\newcommand{\Sjostrand}{Sj\"ostrand }

\newcommand{\C}{\mathbb{C}}

\renewcommand{\H}{\mathbb{H}}

\newcommand{\R}{\mathbb{R}}

\newcommand{\Z}{\mathbb{Z}}

\newcommand{\F}{\mathbb{F}}
\newcommand{\T}{\mathbf{T}}

\newcommand{\acal}{\mathcal{A}}

\newcommand{\ccal}{\mathcal{C}}

\newcommand{\fcal}{\mathcal{F}}

\newcommand{\hcal}{\mathcal{H}}

\newcommand{\lcal}{\mathcal{L}}

\newcommand{\scal}{\mathcal{S}}
\newcommand{\tcal}{\mathcal{T}}

\newcommand{\wh}{\widehat}
\newcommand{\wt}{\widetilde}
\newcommand{\wb}{\overline}
\renewcommand{\rm}{\text}

\newcommand{\bma}{\begin{bmatrix}}
\newcommand{\ema}{\end{bmatrix}}
\newcommand{\baa}{\begin{align*}}
\newcommand{\eaa}{\end{align*}}
\newcommand{\bea}{\begin{eqnarray*} }
\newcommand{\eea}{\end{eqnarray*} }
\newcommand{\bee}{\begin{eqnarray} }
\newcommand{\eee}{\end{eqnarray} }
\newcommand{\be}{\begin{equation} }
\newcommand{\ee}{\end{equation} }
\newcommand{\bp}{\begin{prop}}
\newcommand{\ep}{\end{prop}}
\newcommand{\bt}{\begin{theorem}}
\newcommand{\et}{\end{theorem}}
\newcommand{\bpf}{\begin{proof}}
\newcommand{\epf}{\end{proof}}
\newcommand{\bl}{\begin{lem}}
\newcommand{\el}{\end{lem}}
\newcommand{\bc}{\begin{cor}}
\newcommand{\ec}{\end{cor}}
\newcommand{\bd}{\begin{defin}}
\newcommand{\ed}{\end{defin}}
\newcommand{\bcs}{\begin{cases}}
\newcommand{\ecs}{\end{cases}}
\newcommand{\bex}{\begin{example}}
\newcommand{\eex}{\end{example}}
\newcommand{\brem}{\begin{rem}}
\newcommand{\erem}{\end{rem}}

\renewcommand{\ss}{\subsection}

\newcommand{\isoto}{\xrightarrow{\sim}}

\newcommand{\pa}{\partial}
\newcommand{\ot}{\otimes}
\newcommand{\la}{\langle}
\newcommand{\ra}{\rangle}
\newcommand{\half}{\frac{1}{2}}
\renewcommand{\d}{\partial}
\newcommand{\dbar}{\bar\partial}
\newcommand{\ddbar}{\partial\dbar}
\newcommand{\RM}{\backslash}

\newcommand{\om}{\omega}

\newcommand{\hPi}{\h \Pi}

\newcommand{\h}{\hat} %there are too many hat in. 
\newcommand{\kk}{\left( \frac{k}{2\pi} \right)}
\newtheorem*{maintheo}{{\sc {\bf Main Theorem}}}
\newtheorem{theointro}{{\sc {\bf Theorem}}}
\newtheorem{theo}{{\sc Theorem}}[section]
\newtheorem{cor}[theo]{{\sc Corollary}}
\newtheorem{defin}[theo]{{\sc Definition}}
\newtheorem{rem}[theo]{{\sc Remark}}
\newtheorem{example}[theo]{{\sc Example}}

\newtheorem{lem}[theo]{{\sc Lemma}}
\newtheorem{prop}[theo]{{\sc Proposition}}

\title{Central Limit theorem for spectral Partial  Bergman kernels }

\author{Steve Zelditch and Peng Zhou}
\address{Department of Mathematics, Northwestern  University, Evanston, IL 60208, USA}
\email{zelditch@math.northwestern.edu, \,pzhou.math@gmail.com}
\thanks{Research partially supported by NSF grant and DMS-1541126
and by the Stefan Bergman trust  .}

\begin{document}
\maketitle
%\tableofcontents

\begin{abstract} Partial Bergman kernels $\Pi_{k, E}$ are kernels of orthogonal projections onto  subspaces $\scal_k \subset H^0(M, L^k)$ of holomorphic 
sections of the $k$th power of an ample line bundle over a \kahler manifold $(M, \omega)$. The subspaces of this article are spectral subspaces
$\{\hat{H}_k \leq E\}$ of the Toeplitz quantization $\hat{H}_k$  of a smooth Hamiltonian
$H: M \to \R$. It is shown that the relative partial density of states $\frac{\Pi_{k, E}(z)}{\Pi_k(z)} \to {\bf 1}_{\acal}$ where $\acal  = \{H < E\}$. Moreover it is shown that this partial density of states exhibits `Erf'-asymptotics along the interface $\partial \acal$, that is, the density profile asymptotically has a  Gaussian error function shape interpolating between the values $1,0$ of ${\bf 1}_{\acal}$. Such `erf'-asymptotics are a universal
edge effect. The different types of scaling asymptotics are reminiscent of the law of large numbers and central limit theorem.

\end{abstract}

This article is part of a series \cite{ZZ} devoted to partial Bergman kernels
on polarized \kahler manifolds $(L, h) \to (M^m, \omega, J)$, i.e. \kahler manifolds of (complex) dimension $m$
equipped with a Hermitian holomorphic line bundle  whose curvature form $F_\nabla$ for the Chern connection $\nabla$ satisfies $\omega = i F_\nabla$ 
%\footnote{
%Our normalization convention for $\omega$ is larger by $2\pi$ than that of \cite{ZZ}. Here $\Pi_k(z) = \kk^m(1+\cdots)$ instead of $k^m(1+\cdots)$. 
%}
. Partial Bergman
kernels \begin{equation} \label{PIS} \Pi_{k, \scal_k} : L^2(M, L^k) \to \scal_k \subset H^0(M, L^k) \end{equation}
are Schwarz kernel for orthogonal projections onto proper subspaces $\scal_k$ of the holomorphic sections of $L^k$. For certain sequences $\scal_k$ of subspaces, the partial density of states $ \Pi_{k, \scal_k}(z)$  has an asymptotic expansion as $k \to \infty$ which roughly gives the probability density
that a quantum state from $\scal_k$ is at the point $z$.  More concretely, in terms of  an orthonormal basis $\{s_i\}_{i=1}^{N_k}$ of $\scal_k$, the partial Bergman densities  defined by  \begin{equation} \label{DOS} \Pi_{k,\scal_k}(z) = \sum_{i=1}^{N_k} \|s_i(z)\|^2_{h^k}. \end{equation} When $\scal_k = H^0(M, L^k)$, $\Pi_{k, \scal_k} = \Pi_k: L^2(M, L^k) \to H^0(M, L^k)$ is the orthogonal (\szego
or Bergman) projection. 
We also call the ratio $\frac{\Pi_{k, \scal_k}(z)}{\Pi_k(z)}$ the partial density of states.

Corresponding to $\scal_k$ there is
an allowed region $\acal$ where the relative partial density of states $\Pi_{k, \scal_k}(z) / \Pi_k(z)$ is one, indicating
that the states in $\scal_k$ ``fill up'' $\acal$, and a forbidden region $\fcal$ where the relative density
of states is $O(k^{-\infty})$, indicating that the states in $\scal_k$ are almost zero in $\fcal$.
On the boundary $\ccal: =\partial \acal$ between the two regions there is a
shell of thickness $O(k^{-\half})$  in which the
density of states decays from $1$ to $0$. One of the prinicipal results of this article is that the $\sqrt{k}$-scaled  relative partial
density of states is asymptotically Gaussian along this interface, in a way reminiscent of the central limit theorem. This was proved in \cite{RS} for certain Hamiltonian holomorphic  $S^1$ actions, then in greater generality in
\cite{ZZ}. The results of this article show it is a universal property of partial Bergman kernels defined by $C^{\infty}$ Hamiltonians. 

Before stating our results, we explain how we define the subspaces
$\scal_k$.
 In \cite{ZZ} and in this
article, they are defined as spectral subspaces for the quantization of a smooth function $H: M \to \R$. 
By  the standard (Kostant) method of geometric quantization, one can quantize $H$ as the self-adjoint  zeroth order Toeplitz operator  \begin{equation} \label{TOEP} {H}_k:= \Pi_k (\frac{i}{k} \nabla_{\xi_H} +   H) \Pi_k: 
H^0(M, L^k) \to H^0(M, L^k) \end{equation} 
 acting on the space
$H^0(M, L^k)$ of holomorphic sections.  Here,  $\xi_H$ is the Hamiltonian vector field of $H$, $\nabla_{\xi_H}$ is the Chern covariant deriative on sections,   and $H$ acts by multiplication. Let $E$ be a regular value of $H$. We denote the  partial Bergman kernels for the corresponding spectral
subspaces by
\begin{equation} \label{PBK} \Pi_{k, E} : H^0(M, L^k)
\to \hcal_{k, E}, \end{equation}
where 
\begin{equation} \label{HEintro} \scal_k: = \hcal_{k, E}: = \bigoplus_{\mu_{k,j} < E}
V_{\mu_{k, j}}, \end{equation}
$\mu_{k,j} $ being the eigenvalues of ${H}_k$ and
\begin{equation} \label{EIGSP} V_{\mu_{k,j}}: = \{s \in H^0(M, L^k) : 
{H}_k s = \mu_{k, j} s\}. \end{equation} 
The associated allowed region $\acal$ is the classical counterpart to \eqref{HEintro}, and the forbidden region $\fcal$ and the interface $\ccal$ are 
\begin{equation} \label{acalDef} \acal: = \{z: H(z) < E\},  \quad \fcal = \{z: H(z) >E\}, \quad \ccal = \{z:H(z)=E\}. \end{equation}

For each $z \in \ccal$, let $\nu_z$ be unit normal vector to $\ccal$ pointing towards $\acal$. And let $\gamma_z(t)$ be the geodesic curve such that $\gamma_z(0) = z, \dot \gamma_z(0) = \nu_z$. For small enough $\delta>0$, the map 
\be \Phi: \ccal \times (-\delta, \delta) \to M, \quad (z, t) \mapsto  \gamma_z(t) \label{normalexp}\ee
is a diffeomorphism onto its image.

\begin{maintheo}\label{MAINTHEO}
Let $(L, h) \to (M, \omega, J)$ be a polarized \Kahler manifold. Let $H: M\to \R$ be a smooth function and $E$ a regular value of $H$. Let  $\scal_k \subset H^0(X, L^k)$ be defined as in  \eqref{HEintro}. Then we have the following asymptotics on partial Bergman densities $\Pi_{k, \scal_k}(z)$:
\[   \left( \frac{\Pi_{k, \scal_k}}{\Pi_{k}} \right)(z)  = \bcs
 1& \text{if } z \in \acal \\
0 & \text{if } z \in \fcal
\ecs \mod O(k^{-\infty}).  \]
For small enough $\delta>0$, let $\Phi:\ccal \times (-\delta, \delta) \to M$ be given by \eqref{normalexp}. Then for any $z \in \ccal$ and $t \in \R$, we have
\begin{equation} \label{REM}  \left( \frac{\Pi_{k, \scal_k}}{\Pi_{k}} \right)(\Phi(z,  t/\sqrt{k}))  = \Erf(2\sqrt{\pi} t) + O(k^{-1/2}),  \end{equation}
where $\Erf(x) = \int_{-\infty}^x e^{-x^2/2} \frac{dx}{\sqrt{2\pi}}$ is the Gaussian error function.  
\end{maintheo} 

\brem
We could also choose an interval $(E_1, E_2)$ with $E_i$ regular values of $H$\footnote{It does not matter whether the endpoints are included in the interval, since contribution from the eigenspaces $V_{k, \mu}$ with $\mu = E_i$ are of lower order than $k^m$.}, and define $\scal_k$ as the span of eigensections with eigenvalue within $(E_1, E_2)$. However the interval case can be deduced from the half-ray case $(-\infty, E)$ by taking difference of the corresponding partial Bergman kernel, hence we only consider allowed region of the type in \eqref{acalDef}. 
\erem
\bex
As a quick illustration, one can consider holomorphic function on $\C$ with weight $e^{-|z|^2}$. Fix any $\epsilon \geq 0$, then we may define the subspaces $\scal_k = \oplus_{j \leq \epsilon k} z^j$ which are spanned by sections vanishing to order at most $\epsilon k$ at $0$, or sections with eigenvalues $\mu < \epsilon$ for operator $ H_k = \frac{1}{ik} \pa_\theta $ quantizing $H=|z|^2$. The full and partial Bergman densities are
\[ \Pi_k(z) = \frac{k}{2\pi}, \quad \Pi_{k, \epsilon}(z) = \kk \sum_{j \leq \epsilon k} \frac{k^j}{j!} |z^j|^2 e^{-k |z|^2}, \]
As $k \to \infty$, we have
\[ \lim_{k \to \infty} k^{-1} \Pi_{k, \epsilon}(z) = \bcs 1 & |z|^2 < \epsilon \\0 & |z|^2 > \epsilon. \ecs \]
For the boundary behavior, one can consider sequence $z_k$, such that $ |z_k|^2 = \epsilon(1+k^{-1/2} u),$ 
\[\lim_{k \to \infty} k^{-1} \Pi_{k, \epsilon}(z_k) =  \Erf(u). \]
\eex

This example is often used to illustrate the notion of  `filling domains' in the IQH (integer Quantum Hall) effect. 
The following image of the density profile  is copied from \cite{W}, 

%in which the author writes  (paraphrasing pages 285-287)
%
%{\it The zeroth LLL state of angular momentum $m$ is 
%$$\psi_m = \sqrt{N_m} z^m e^{- |z|^2/4 \ell_B^2}. $$
%Then $|\psi_m(z)|^2 $ concentrates in the $m$th ring  centered
%at $|z| = r_m = \sqrt{2 m} \; \ell_B$. The ring has area $
%= 2 \pi  \ell_B^2$ so the m states fill out an area of $\pi r_m^2 = 2 \pi m \ell_B^2$. Then the density profile where the first m levels are filled  is illustrated by }
%\begin{figure}
\begin{center} 
\includegraphics[scale=0.3]{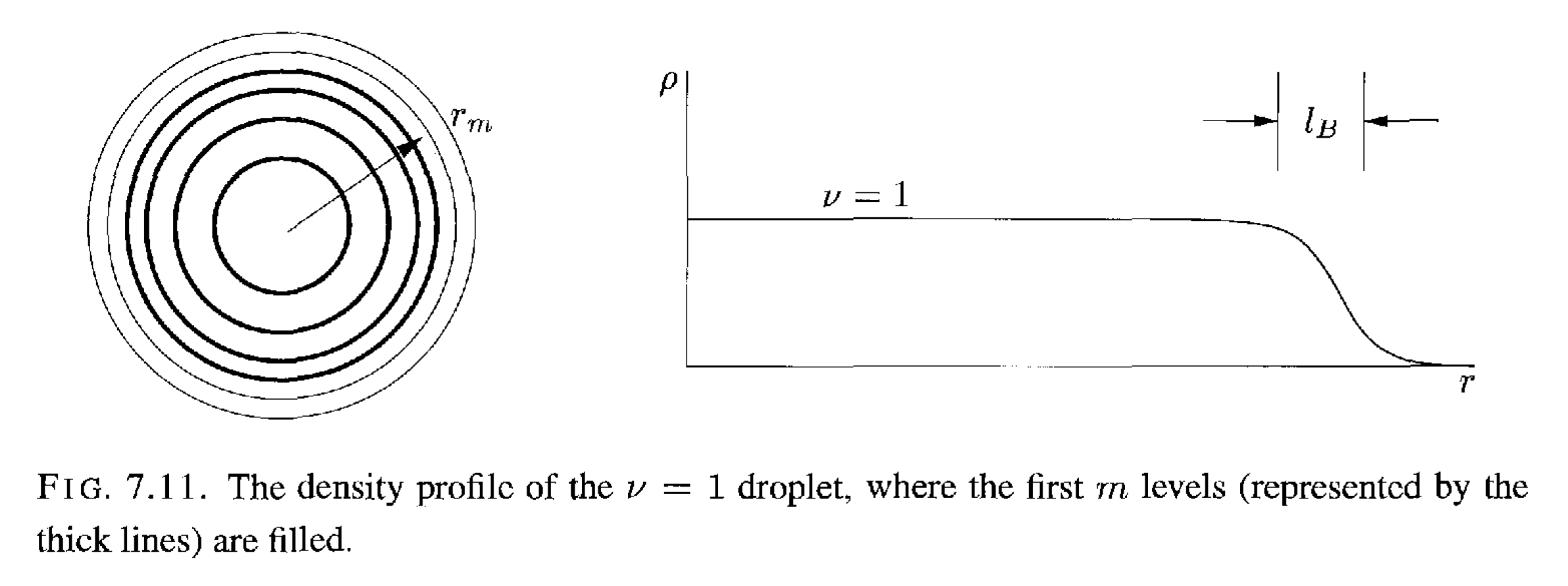}  
\end{center}
%\end{figure}
The graph is that of the Erf (Gaussian error function) or a closely related cousin.  The example is $S^1$ symmetric and therefore the simpler results of \cite{ZZ} apply. For more general domains $D$, even in dimension one, it is not obvious to to fill $D$ with LLL states.  The Main  Theorem answers the question when $D = \{H \leq E\}$ for some $H$. Other approaches are discussed in Section \ref{QHSECT}.

Erf asymptotics are now a standard feature of the IQH and are quite different from the density profile for the fractional QH effect, which is not given by a partial Bergman kernel (see   \cite{Wieg,CFTW} for the comparison of erf-asymptotics for the IQH and the unknown interface asymptotics for
the fractional QY effect.)

\subsection{Three families of measures at different scales}

The rationale for viewing the Erf asymptotics of scaled partial Bergman kernels along the interface $\ccal$ is explained by considering three different scalings of the spectral problem.

\begin{equation} \label{mukzdef} \left\{ \begin{array}{ll} 
(i) &  d\mu_k^z(x)\;\;
=  \sum_{j}
 \Pi_{k,j}(z) \delta_{\mu_{k,j}}(x), \\&\\

(ii) &   d\mu^{z, \half}_{k} (x)
=   \sum_{j}  
  \Pi_{k,j}(z)  \delta_{ \sqrt{k} (\mu_{k,j} - H(z))} (x),  \\  &\\
  
(iii) &   d\mu^{z, 1, \tau}_{k} (x)
=   \sum_{j}  
  \Pi_{k,j}( z)  \delta_{k (\mu_{k,j} - H(z)) + \sqrt{k} \tau} (x), 
  \end{array} \right.  \end{equation}
  where as usual,  $\delta_y$ is the Dirac point mass at $y \in \R$.
We use $\mu(x) = \int_{-\infty}^x d \mu(y)$ to denote the cumulative distribution function. 
  
 We view these scalings as corresponding to three scalings of the  convolution powers $\mu^{*k}$ of a probability measure $\mu$ supported on $[-1,1]$ (say). The third scaling (iii) corresponds to $\mu^{*k}$, which is supported on $[-k,k]$. The first scaling (i) corresponds to 
 the Law of Large Numbers, which rescales $\mu^{*k}$ back to $[-1,1]$. The second scaling (ii) corresponds to the CLT (central limit theorem) which rescales the measure to $[-\sqrt{k}, \sqrt{k}]$. 
  
Our main results give asymptotic formulae for integrals of test functions
and characteristic functions
against these measures. To obtain the remainder estimate \eqref{REM}, we need to apply semi-classical Tauberian theorems to 
%$d\mu_k^{z,\half}$ 
$ \mu^{z, \half}_{k}$ and that forces us to find asymptotics for %$d\mu_k^{z,1}$.
 $\mu^{z, 1,\tau}_{k}$.  

\subsection{Unrescaled bulk results on $d\mu_k^z$}

The first result is that the behavior of the partial density of states in the allowed region $\{z: H(z) < E\}$ is essentially the same as for
the full density of states, while it is rapidly decaying outside this region.

We begin with a simple and general result about partial Bergman kernels for smooth
metrics and Hamiltonians. 

\begin{theointro}\label{AF}  Let $\omega$ be a $C^{\infty}$ metric on $M$ and let $H \in C^{\infty}(M)$.  Fix a regular value $E$ of $H$ and let $\acal, \fcal, \ccal$ be given by \eqref{acalDef}. 
Then for any $f \in C^\infty(\R)$, we have
\be
\Pi_k(z)^{-1} \int_{-\infty}^E f(\lambda) d\mu_k^z(\lambda) \to \bcs
f(H(z)) &  \rm{if } z \in \acal\\
0 & \rm{if } z \in \fcal.\\
\ecs
\label{TOEPeq1}
\ee
In particular,  the  density of states of the partial Bergman kernel is given by the asymptotic
formula:
\be
\Pi_k(z)^{-1} \Pi_{k, E}(z) \sim 
\bcs
1  \mod O(k^{-\infty})  & \rm{if } z \in \acal\\
0 \mod O(k^{-\infty})  &  \rm{if } z \in \fcal.\\
\ecs 
\label{TOEPeq2}
\ee where the asymptotics are uniform on compact sets of $\acal$ or $\fcal.$
\end{theointro}

In effect, the leading order asymptotics says that
$d\mu_k^z \to \delta_{H(z)}$. %. Since  $\int_P x d\mu_k^z(x) = H(z)$, 
This is a kind of Law of Large Numbers for the sequence $d\mu_k^z $.
The theorem does not specify the behavior of $\mu_k^z(-\infty, E) $
when $H(z) =  E$. The next result pertains to the edge behavior.

\subsection{$\sqrt{k}$-scaling results on $d \mu^{z, 1/2}_{k}$}

 The most interesting  behavior occurs in  $k^{-\half}$-tubes around the interface $\ccal$ between the allowed region $\acal$ and the forbidden region $\fcal$.
 For any $T>0$, the tube
of `radius' $T k^{-\half}$ around $\ccal = \{H = E\}$ is the flowout of $\ccal$ under the gradient flow  of $H$ 
\[ F^t := \exp(t \nabla H) : M \to M, \]
for $|t| < T k^{-1/2}$. 
Thus it suffices to study the partial density of states $\Pi_{k,E}(z_k)$ at points $z_k = F^{\beta/\sqrt{k}}(z_0)$ with $z_0 \in H^{-1}(E). $
The interface result for any smooth Hamiltonian is the same as if the Hamiltonian flow generate a holomorphic $S^1$-actions, and thus our result shows that it is
a universal scaling asymptotics around $\ccal$.

\begin{theointro} \label{RSCOR} 
 Let $\omega$ be a $C^{\infty}$ metric on $M$ and let $H \in C^{\infty}(M)$.  Fix a regular value $E$ of $H$ and let $\acal, \fcal, \ccal$ be given by \eqref{acalDef}. Let $F^t: M \to M$ denote the gradient flow of $H$ by time $t$. We have the following results: 
 \begin{enumerate}
 \item For any point $z \in \ccal$, any $\beta \in \R$, and any smooth function $f \in C^\infty(\R)$, there exists a complete asymptotic expansion,
\be \label{thm-2-1}  \sum_{j} f(\sqrt{k}(\mu_{k,j} - E))  \Pi_{k,j}( F^{\beta/\sqrt{k}} (z))   \simeq  \kk^m(  I_0  + k^{-\half}
I_{1} + \cdots) , \ee
in descending powers of $k^{\half}$, with the leading coefficient as
\[I_0(f, z, \beta) =   \int_{-\infty}^\infty f(x) e^{-\left(\frac{x  }{|\nabla H|(z)|} - \beta |\nabla H(z)| \right)^2} \frac { dx}{\sqrt{\pi}  |\nabla H (z)|}.\]

\item For any point $z \in \ccal$, and any $\alpha \in \R$, the cumulative distribution function $\mu^{z,1/2}_k(\alpha)=\int_{-\infty}^\alpha d \mu^{z,1/2}_k$ is given by
\be \mu^{z,1/2}_k(\alpha)  = \sum_{\mu_{k,j} < E + \frac{\alpha}{\sqrt{k}}} \Pi_{k,j}(z) =   \kk^m \Erf\left( \frac{ \sqrt{2} \alpha}{|\nabla H(z)|} \right) + O(k^{m-1/2}). \label{eqmuhalf} \ee

\item For any point $z \in \ccal$, and any $\beta \in \R$, the Bergman kernel density near the interface is given by
\begin{equation} \label{REMEST}  \Pi_{k,E}(F^{\beta/\sqrt{k}}(z))=  \sum_{\mu_{j,k} <  E}  \Pi_{k,j}( F^{\beta/\sqrt{k}} (z)) =  \kk^m\Erf \left( -\sqrt{2}\beta |\nabla H(z)| \right)+ O(k^{m-1/2}).\end{equation}

 \end{enumerate}

\end{theointro}

\begin{rem}
The  leading power $\kk^m$ is the same as in Theorem \ref{AF}, despite the
fact that we sum over a packet of eigenvalues of width  (and cardinality) $k^{-\half}$ times the width (and cardinality) in Theorem
\ref{AF}. This is because the summands $ \Pi_{k,j}(z)$ already
localize the sum to $\mu_{k,j}$ satisfying $|\mu_{k,j} - H(z)| < C k^{-\half}$. 
\end{rem}

%
%\begin{cor} We have,
%\begin{equation} \label{CENTERED} \lim_{k \to \infty} k^{-m}  \sum_{j: \mu_{k,j} \in P_0} f(\sqrt{k}(\mu_{k,j} - H(z))) \Pi_{k,j}(z) = \int_{-\infty}^\infty f(x) e^{-\frac{1}{2} \left(\frac{2x \sqrt{\pi}}{|\nabla H(z) |}\right)^2} \frac {2 dx}{\sqrt{2  |\nabla H| (z)}}.\end{equation}
%\end{cor}
%\bpf
%Let $E = H(z)$ and $\beta=0$ in Theorem \ref{RSCOR}.
%\epf

%Here, for $h >0$ define the  dilation operator $D_h f(x) := f(hx). $ 
% Dually, define the dilation 
%  $\mu_h = D^*_h \mu$ on measures by 
%$\mu_h[a,b] = \mu[h^{-1} a, h^{-1} b]$.  Note that $D^*_h$ is the adjoint of $D_h$
%under the pairing $\lagle f, \mu \ragle$ of $C(\R)$ with $\mcal(\R)$ (positive measures), i.e.
%$\lagle D_h f, \mu \ragle = \lagle f, D_h^* \mu \ragle. $ Formally,
%$d\mu_h(x) = d\mu(\frac{x}{h})$. Also, let $\tau_E f(x) = f(x - E)$. 
% See Section \ref{muhalfsect} for the proof of this re-statement.

\subsection{Energy level localization and $d \mu^{z, 1,\alpha}_{k}$}

To obtain the remainder estimate for the $\sqrt{k}$ rescaled measure $d\mu^{z,1/2}_k$ in \eqref{eqmuhalf} and \eqref{REMEST} , we apply the Tauberian theorem. Roughly speaking, one approximate $d\mu^{z,1/2}_k$ by convoluting the measure with a smooth function $W_h$ of width $h$, and the difference of the two is proportional to $h$. The smoothed measure $d \mu^{z,1/2}_k * W_h$ has a density function, the value of which can be estimated by an integral of the propagator $U_k(t, z, z)$ for $|t| \sim k^{-1} / (h k^{-1/2})$. 
%This is the uncertainty principle, with $\Delta E \Delta t \sim \hbar$, with $\hbar = k^{-1}$, the energy uncertainty $\Delta E \sim h k^{-1/2}$, hence the time uncertainty is $\Delta t = \hbar / \Delta E = k^{-1/2}/h$. 
Thus if we choose $h = k^{-1/2}$, and $W_h$ to have Fourier transform supported in $(-\epsilon, +\epsilon)/h$, we only need to evaluted $U_k(t, z,z)$ for $|t| < \epsilon$, where $\epsilon$ can be taken to be arbitrarily small.

\begin{theointro} \label{ELLSMOOTH} 
Let $E$ be a regular value of $H$ and $z \in H^{-1}(E)$. If $\epsilon$ is small enough, such that the Hamiltonian flow trajectory starting at $z$ does not loop back to $z$ for time $|t| < 2\pi \epsilon$, then for any Schwarz function 
 $f \in \scal(\R)$ with $\hat{f}$ supported in $(-\epsilon, \epsilon)$ and $\h{f}(0) = \int f(x) dx = 1$, and for any $\alpha \in \R$ we have 
\[
\int_\R f(x) d \mu^{z, 1, \alpha}_k(x) = \kk^{m-1/2} e^{- \frac{\alpha^2}{\|\xi_H(z)\|^2}} \frac{\sqrt{2}}{2\pi \|\xi_H(z)\|}(1+ O(k^{-1/2})).
\]
\end{theointro}

%\edit{Have to work out right one}
% This spreads out the support of the measure to $k P$. In the probabilistic analogy it plays the role of the kth convolution power of the measure for $k = 1$.

It is an interesting and well-known problem, in other settings, to find the asymptotics when
$\rm{supp}(\hat{f})$ is a general interval. 
%\begin{equation} \label{RHOSQRTintrob} \mu^z_{k, E}(f): =\sum_{j} f(k %(\mu_{k,j} - E)) \Pi_{k, j}(z) =   \int_{\R} \hat{f}( t) e^{- i E k  t} \hat{\Pi}_{h^k} %\sigma_{k t} (\hat{g}^t)^* \hat{\Pi}_{h^k}(z) dt \end{equation}
%It is in line with prior results on pointwise asymptotics of Weyl sums that the asymptotics of $\mu^z_{k, E}(f)$ are trivial unless $z \in H^{-1}(E)$. 
When $z \in H^{-1}(E)$ they depend on whether or not $z$ is a periodic
point for $\exp t \xi_H$. In this article we only need to consider the singularity
at $t = 0$, or equivalently test functions for which the support of $\hat{f}$ is sufficiently close to $0$. In \cite{ZZ17}, the long time dynamics of the Hamiltonian flow  are used to give a complete asymptotic expansion for \eqref{RHOSQRTintrob} when $f \in C^{\infty}_c(\R)$ and a  two term asymptotics with remainder   when $f = {\bf 1}_{[E_1, E_2]}$.

\subsection{Sketch of the proofs} 
 
The theorems are proved by first  smoothing the sharp interval cutoff ${\bf 1}_{[E_{min}, E]}$  to a smooth cutoff  $f$  and obtaining asymptotics, 
and then applying a Tauberian argument. The jump discontinuity of  ${\bf 1}_{[E_{min}, E]}$   produces the universal
error function transition between the allowed and forbidden regions. This error function arises in classical approximation
arguments involving Bernstein polynomials (see \cite{ZZ} for background and references).
This is a standard method in proving
 sharp pointwise Weyl asymptotics by combining smoothed asymptotics with Tauberian theorems.

Given a function $f \in \scal(\R)$ 
(Schwartz space) one defines \begin{equation} \label{UNSCALED}
f( {H}_k) = \int_{\R} \hat{f}(\tau) e^{i \tau
{H}_k } \frac{d\tau}{2\pi} 
\end{equation} to be the operator  on $H^0(M, L^k)$ with the same eigensections
as ${H}_k$ and with eigenvalues $f(\mu_{k,j})$. 
Thus, if $s_{k,j}$ is an eigensection of ${H}_k$, then
\begin{equation} f({H}_k)  {s}_{k,j} =
f\left(\mu_{k,j}\right)  {s}_{k,j}
\end{equation} 

Let $E_{min}, E_{max}$ be such that $H(M) = [E_{min}, E_{max}]$. Given a regular value of $E$ of $H$, the subspace $\scal_k$ in \eqref{HEintro} is defined as the range of
$f({H}_k) $ where
 $f = {\bf 1}_{[E_{min}, E]}$ and the partial density of states is given by the metric contraction of the kernel,
\begin{equation} \Pi_{k,E}(z ) = f({H}_k)(z) =  \sum_{j: \mu_{k,j}\leq E}
\Pi_{k,j}(z). \end{equation}
For a smooth test function
$f$, $\Pi_{k,f}(z)$ is  
 the metric contraction of the Schwartz kernel of $f({H}_k)$  at $z = w$, is given by 
%\begin{equation} f({H}_k) \hat{\Pi}_{h^k}(e^{i \phi} \hat{z}, \hat{w})
%= \sum_{j=1}^{N_k} f\left(\frac{j}{k} \right) \Pi_{k, j}(z, w),
%\end{equation}
%where $ \Pi_{k, j}(z, w)$ is the spectral projection \eqref{SPECPROJ}.
%It follows that
\begin{equation} \label{fPBK} \Pi_{k,f}(z)
= \sum_{j}
f\left(\mu_{k,j}\right) \Pi_{k,j}(z).
 \end{equation}

Note that $\Pi_k e^{i t {H}_k} \Pi_k$ is the exponential
of a bounded Toeplitz pseudo-differential operator $H_k$ and itself is a Toeplitz pseudo-differential operator. To obtain a dynamical
operator, i.e. one which quantizes a Hamiltonian flow,  one needs to exponentiate the first order Toeplitz operator $k {H}_k$. In that case,
 \begin{equation} \label{fofHFT}
f(k ({H}_k-E)) = \int_{\R} \hat{f}(\tau) e^{i k \tau
{H}_k } \frac{d\tau}{2\pi} =  \int_{\R} \hat{f}(\tau) U_k(\tau) \frac{d\tau}{2\pi},
\end{equation} 
where  \begin{equation} \label{UkDEF} U_k(\tau) = \exp( i \tau k {H}_k). \end{equation} is the unitary group on $H^0(M, L^k)$ generated by $k {H}_k$.

In \S \ref{TQD} it is shown that $U_k(t)$ is a semi-classical Toeplitz Fourier integral operator of a type defined
in \cite{Z97}. To construct a semi-classical  parametrix for $U_k(t)$ it is convenient to lift the Hamiltonian flow $g^t$
of $H$
to a contact flow $\hat{g}^t$ on the unit circle bundle $X_h = \{\zeta \in L^*: h(\zeta) = 1\}$ associated to the Hermitian metric $h$ on $L^*$.  That is, we lift sections
$s$ of $L^k$ to equivariant functions $\hat{s}: X_h \to \C$ transforming by $e^{i k \theta}$ under the
natural $S^1$ action on $L^*$. Holomorphic sections lift to CR holomorphic functions and the space
$H^0(M, L^k)$ lifts to the space $\hcal_k(X_h)$ of equivariant CR functions. The orthogonal projection
$\Pi_k$ onto $H^0(M, L^k)$ lifts to the orthogonal projection $\hat{\Pi}_{h^k} $ onto $\hcal_k(X_h)$.
In Proposition \ref{SC} it is shown that the lift  $\hat{U}_k(t)$ to $\hcal_k(X_h)$ has the form,
$\hat{\Pi}_{h^k}  (\hat{g}^t)^* \sigma_{k,t} \hat{\Pi}_{h^k}$ where $(\hat{g}^t)^*$ is the pullback of functions on $X_h$
by $\hat{g}^t$ and where $\sigma_{k,t}$ is a semi-classical symbol.

The main tool in the proof of Theorem  \ref{RSCOR} is to use the Boutet-de-Monvel-\Sjostrand parametrix
to study the integrals $\int_{\R} \hat{f}(t) U_k(\frac{t}{\sqrt{k}}, z,z)d t$. Since the relevant time interval
is `infinitesimal' (of the order $k^{-1/2}$) the result can be proved by linearizing the kernel $ U_k(\frac{t}{\sqrt{k}}, z,z)$.
The smoothed interface asymptotics of  thus amount to the asymptotics of  the dilated sums,
\begin{equation} \label{RHOSQRTintro}\sum_{j} f(\sqrt{k} (\mu_{k,j} - E)) \Pi_{k, j}(F^{\beta/\sqrt{k}} (z_0))  =   \int_{\R} \hat{f}( t) e^{- i E \sqrt{k} t} \hat{\Pi}_{h^k}  (\hat{g}^t)^* \sigma_{k,t} \hat{\Pi}_{h^k}(F^{\beta/\sqrt{k}} (z)) \frac{dt}{2\pi} \end{equation}
where   $z \in \partial \acal = H^{-1}(E)$  and where $\hat{f} \in L^1(\R)$, so that
the integral on the right side converges.
We  employ the Boutet-de-Monvel-Sjostrand parametrix to give an explicit formula for the right side of \eqref{RHOSQRTintro} modulo
small remainders. 

At this point, we are essentially dealing with $\frac{1}{\sqrt{k}}$ scaling asymptotics of \szego kernels,  as studied first in \cite{BSZ},
then in more detail in \cite{ShZ02, MaMa07, LuSh15} and for related dynamical purposes by Paoletti in \cite{P12, P14}. The scaling
asymptotics of the Bergman kernel infinitesimally off the diagonal  at $z_0$ are expressed in terms of the osculating Bargman-Fock-Heisenberg kernel
$\Pi_{BF}^{T_{z_0} M}$ of the (complex) tangent space at $z_0$. This tangent space is equipped with a complex structure $J_z$
and a Hermitian metric $H_z$ and therefore with a Bargmann Fock space $H^2(T_{z_0} M, \gamma_{J_{z_0}, H_{z_0}})$ of entire $J_{z_0}$-holomorphic
functions on $T_{z_0} M$ which are in $L^2$ with respect to the Gaussian weight determined by $J_z, H_z$. This is the linear model 
for  semi-classical Toeplitz calculations. 
Since we reduce the general calculation of scaling asymptotics to the linear ones, we present the calculations in the Bargmann-Fock model in detail first (see Sections \ref{BFmodel1}-\ref{BFmodel2}-\ref{BFmodel3}).
 We emphasize that the linear model is not only an example, but constitutes a fundamental part of the proofs.

To prove Theorem \ref{ELLSMOOTH}, we study the integrals
\begin{equation} \label{RHOSQRTintrob} \mu^{z,1, \tau}_{k}(f): =\sum_{j} f(k (\mu_{k,j} - H(z))+\sqrt{k} \tau) \Pi_{k, j}(z) =   \int_{\R} \hat{f}( t) e^{- i H(z) k  t + i \sqrt k \tau t} \hat{\Pi}_{h^k}  (\hat{g}^t)^* \sigma_{k,t} \hat{\Pi}_{h^k}(z) \frac{dt}{2\pi} \end{equation}
For purposes of this article, we only need the infinitesimal time behavior of the Hamilton flow. In \cite{ZZ17}  we use the long time asymptotics of the Hamilton flow to obtain a two term Weyl law.

After obtaining scaling asymptotics for smoothed partial densities of states, we employ  Tauberian theorems
with different scalings 
 to obtain  asymptotics with remainders for the sharp partial densities of states  or the
 measures \eqref{mukzdef}. Background on  Tauberian theorems is given  in \S \ref{TAUBERAPP} and \S \ref{TAUBER}.

\subsection{\label{S1COMP} Comparison to the $S^1$ case}

For purposes of comparison, we review the results of \cite{RS, ZZ}
in the case where the Hamiltonian generates an $S^1$ action by holomorphic maps.

In  \cite{RS, ZZ}, it is assumed that $(M, L, h)$ is invariant under a  Hamiltonian holomorphic $S^1$ action. The action naturally quantizes
or linearizes on the spaces $H^0(M, L^k)$, and  $\scal_k$ is defined in terms of the weights (eigenvalues) of the $S^1$ action on the spaces $H^0(M, L^k)$. The eigenvalues of the generator ${H}_k$ of the quantized
$S^1$ action are `lattice points' $\frac{j}{k}$ where $j \in \Z$. The partial
Bergman kernel $\Pi_{k,j}$ onto a single weight space $V_k(j) \subset H^0(M, L^k)$  already has asymptotics as $k \to \infty, \frac{j}{k} \to E$ and the asymptotics of the
partial Bergman kernels corresponding to intervals of weights is obtained
by integrating these `equivariant' Bergman kernels.

The  `equivariant Bergman kernels'
$\Pi_{k, j}: H^0(M, L^k) \to V_k(j)$  resemble transverse Gaussian
beams along the energy level $H^{-1}(\frac{j}{k})$, and  have complete asymptotic
expansions that can be summed over $j$ to give asymptotics of the density of states $\Pi_{h^k, [E_1, E_2]}(z,z)$ of partial Bergman kernels as $k \to \infty$. In the allowed region $H^{-1}([E_1, E_2]),$
$k^{-m} \Pi_{h^k, [E_1, E_2]}(z,z) \simeq k^{-m} \Pi_k(z,z) \simeq 1$
while in the complementary forbidden region the asymptotics are rapidly
decaying, and exponentially decaying when the metric $\omega$ is real 
analytic. 
The equivariant kernels additionally possess Gaussian scaling asymptotics along the
energy surface, and by summing in  $j$ one obtains   incomplete Gaussian asymptotics of the partial
Bergman kernels along the boundary  $H^{-1}\{E_1, E_2\}$ as in Ross-Singer \cite{RS}.

In the non-periodic case of this article, 
the Hamilton flow    $$g^t := \exp t \xi_H: M \to M, $$
generated by  $H: M \to \R$ with respect to $\omega$ is not holomorphic.  Here and below, 
 we use the notation $\exp t X$ for the flow of a real vector field $X$.
The  gradient flow is denoted by
\begin{equation} \label{gradflow} F^t: = \exp t \nabla H: M \to M. \end{equation} This change from a holomorphic Hamiltonian $S^1$ action to a Hamiltonian $\R$ action brings many new features into the asymptotics
of partial Bergman kernels. First, 
the gradient flow of $\nabla H$
no longer commutes with the Hamilton flow of $\xi_H$, so that one does
not have a global $\C^*$ action to work with.  As mentioned above,  the eigenvalues
generally have multiplicity one and the  eigenspace projections
\begin{equation} \Pi_{k, \mu_{k,j}}: H^0(M, L^k) \to V_{\mu_{k,j}} \end{equation} do not have individual asymptotics. One only obtains
asymptotics if one sums over a `packet' of $k^{m-1}$ eigenvalues in
a `smooth' way.  The asymptotics of this counting function are given
by the reasonably well-understood 
Weyl law for semi-classical Toeplitz operators.
When $H$ is real analytic,   exponential decay of density of states for $z \in \fcal$ still occurs but the rate has a different shape from that in the holomorphic case.  But as shown in this article, the  interface asymptotics of \cite{RS,ZZ} are universal, and continue to hold even
in the $C^{\infty}$ case.

 \begin{rem} The corresponding result for holomorphic Hamiltonian  $S^1$ actions (Theorem  0.4 of \cite{ZZ}) states the
 following. Let 
     $\omega$ be  a $C^{\infty}$ $\T$-invariant \kahler metric, 
and let   $H$ generates the holomorphic $\T$ action.
Fix $E \in H(M)$, and let $z = e^{\beta/\sqrt{k}} \cdot z_0$ for some $z_0 \in H^{-1}(E) \in \R$, 
Above, $e^{\beta/\sqrt{k}} $ denotes the imaginary time part of the $\C^*$ action generated by $H$. The  gradient flow \eqref{gradflow}  of this article is the same as the imaginary time Hamiltonian flow $e^{\beta/\sqrt{k}} $ in the $S^1$ case.
Moreover,  in the $S^1$ case, $H =  \half \partial_{\rho}  \phi, \;\; \pi \partial^2_{\rho} \phi = |\nabla H|^2 $.
The asymptotics also hold if $f = {\bf 1}_{[E_1, E_2]}$ is the characteristic
function of an interval and for  $f(x) = 1_{[0, \infty]}(x)$, then 
\[k^{-m} \sum_{j > k E}  \Pi_{k,j}( e^{\beta/\sqrt{k}} \cdot z_0) =  \int_{-\infty}^{\beta \sqrt{\pa_\rho^2\varphi(z_0)}} e^{-t^2/2} \frac{dt}{\sqrt{2\pi}} + O(k^{-1/2}).\]
Hence this agrees with the formula of Theorem \ref{RSCOR}.

\end{rem}

\subsection{\label{QHSECT} Other approaches and related results}

As mentioned above, this article is partly  motivated by the somewhat vague question, {\it How do you
fill a domain $D$ with quantum states in the LLL (lowest Landau level)?} 
Roughly speaking, if the states are denoted $\psi_j$, then we want
$\psi_j$ to be in the LLL (i.e. zero modes for the Landau Hamiltonian, and therefore holomorphic) and
we want $\frac{1}{N}\sum_{j = 1}^N |\psi_j(z)|^2 \simeq {\bf 1}_D$ to be approximately
equal to the characteristic function ${\bf 1}_D$ of $D$. There are a number
of inequivalent  ways  in which this question could be  formulated precisely.
The approach of this article is to represent $D = \{H \leq E\}$ for some
smooth Hamiltonian $H$ and to prove that the eigensections of $\hat{H}_k$
\eqref{TOEP} 
from the spectral subspace $\hcal_{k, E}$  \eqref{HEintro} asymptotically fill $D$. The Main Theorem and 
Theorem \ref{AF}  show that,  as $k \to \infty$,
$$k^{-m} \Pi_{h^k, [E_1, E_2]}(z,z) \to {\bf 1}_{H^{-1}[E_1, E_2]} $$
and thus ``fills the domain $H^{-1}[E_1, E_2]$ with lowest Landau levels''
without spilling outside the domain. 

Another approach is to to use  $\Pi_k {\bf 1}_D \Pi_k$,
a Toeplitz operator with discontinuous symbol $ {\bf 1}_D$. To leading order it agrees with $\Pi_{k, E}$ when $D = \{H \leq E\}$, but it is not a projection operator. Its eiegensections with eigenvalues close to $1$ should be the states which fill up $D$. It would be interesting to compare this operator to $\Pi_{k, E}$ more precisely.

  A third,  geometrically interesting,  way to define $\scal_k$  is to introduce  a smooth integral divisor $D$ in $X$, and let the subspace $\scal_k$ to be holomorphic sections in $L^k$ that vanishing at least to order $\lfloor \epsilon k \rfloor$ along $D$, where $\epsilon$ is a small enough positive number. The allowed region is given by 
\[ \acal = \{\phi_{\eq,\epsilon,D}(z)=0\}, \; \text{ where} \;\;
\phi_{\eq,\epsilon,D} (z)= \sup \{ \wt \phi(z) : \omega + i \ddbar \wt \phi \geq 0, \;  \wt \phi \leq 0, \; \nu(\wt \phi)_w \geq \epsilon \, \forall w \in D \} \]
and $\nu(\wt \phi)_w$ is the Lelong number at $w$ (See \cite{Ber1} Section 4). The boundary behavior of the partial Bergman kernels  for such $\scal_k$ is only known in the case where $D$ is fixed by a holomorphic $\C^*$-action on $M$, basically because in this case it coincides with the spectral theory
of the Hamiltonian generating the underlying $S^1$ atction \cite{RS,ZZ}.

\section{Background}

The background to this article is largely the same as in \cite{ZZ}, and we refer there for many details.  Here we give a lightning review to setup the notation. First we introduce co-circle bundle $X \subset L^*$ for a positive Hermitian line bundle $(L,h)$, so that holomorphic sections of $L^k$ for different $k$ can all be represented in the same space of CR-holomorphic functions on $X$, $\hcal(X) = \oplus_k \hcal_k(X)$. Then we define the \Szego projection kernel
and state the Boutet de Monvel-\Sjostrand parametrix. In the end, we give Bergman kernel for Bargmann-Fock model on $\C^n$. 

\ss{Positive Line bundle $(L,h)$ and the dual unit circle bundle $X$} \label{ss:circle}
Let $(M, \omega, J, g)$ be a \Kahler manifold, where $\omega$ is a $J$-invariant symplectic two-form and $g$ is a Riemannian metric determined by $\omega, J$ as $g(- ,-) = \om(-, J -)$. Let $(L, h)$ be a holomorphic Hermitian line bundle on $M$. Let  $e_L \in \Gamma(U, L)$ be an local frame over an open subset $U \subset M$, and the  Kahler potential $\varphi(z)$ over $U$ is defined by $e^{-\varphi(z)}=h(e_L, e_L)$. The local expression for the Chern connection $\nabla$ and Chern curvature $F_\nabla$ are $\nabla = \d + \dbar - \d \varphi \wedge (\cdot)$ and $F_\nabla = \ddbar \varphi$. We say $(M, \omega)$ is polarized by $(L,h)$ if the \Kahler form $\omega$ is
\[ \omega =  2\pi c_1(L) = i F_\nabla = i \ddbar \varphi = -\frac{1}{2} dd^c \varphi, \]
where $d^c = i(\d - \dbar)$ such that $\la d^c f, - \ra = \la df, J (-) \ra$. 

Let $(L^*, h^*)$ be the dual bundle to $L$ with the induced hermitian metric $h^*$, and $e_L^* \in \Gamma(U, L^*)$  the dual frame to $e_L$ with $\|e_L^*\| = 1/\|e_L\| = e^{\varphi(z)/2}$. The unit open disk bundle is $D=D(L^*, h^*) = \{p \in L^*, \|p\| < 1\} $ and the unit circle bundle is  $X=\pa D$. Let $\pi: X \to M$ be the projection, then there is a canonical circle action $r_\theta$ on $X$. 
Let $\rho$ be a smooth function defined in a neighborhood of $X$, such that $\rho > 0$ in $D$,  $\rho|_X=0$  and $d \rho|_X \neq 0$. In this paper, we fix a choice of $\rho$ as
\[ \rho(x) = - \log \|x\|_h^2 =  - 2\log | \lambda | - \log \|e_L^*\|^2 =  - 2\log | \lambda | - \varphi(z)  \]
where $x = \lambda e^*_L(z) \in L^*|_U, z \in U$ and $\lambda \in \C^*$. 
Then $X$ can be equipped with a contact one-form $\alpha$,  
%$- i \dbar \log \bar \lambda = - i d (r - i \theta) = - d\theta$
\begin{equation} \label{alphadef} \alpha =- \Re(i\dbar\rho)|_X =   d \theta + \pi^* \Re( i \dbar \varphi(z))  =  d \theta   - \frac{1}{2}  d^c \varphi(z),  \quad \text{ and } d \alpha =  \pi^* \omega\end{equation} 
where $(z, \theta)$ is a local coordinate on $X$, given by 
\be (z, \theta) \mapsto e^{i \theta} \cdot e_L^*(z)/\|e_L^*(z)\| = e^{i \theta - \varphi(z)/2} e_L^*(z). \label{Xcor}\ee The Reeb vector field on $X$ is given by $R = \pa_\theta. $

%
%\bex
% On $\C$, with standard \Kahler potential $\varphi(z) =  |z|^2$, then $\omega = i\ddbar \varphi = i dz \wedge d\bar z = 2 dx \wedge dy$. We have 
%\[ \alpha = d\theta - \frac{i}{2} (\bar z dz - z d \bar z) =  d\theta +  x dy - y dx. \]
%\eex

{\bf Naming convention:} For points in the base space $M$, we use names as $z, w, \cdots$; for points in the circle bundle $X$, we use $\h z, \h w, \cdots$, such that under the projection $\pi: X \to M$, $\pi(\h z) = z, \pi(\h w) = w$, etc. In general, objects upstairs in the circle bundle $X$ with a corresponding object in $M$ are equipped with a hat ``$\wh{\;\;}$''. 

\ss{Holomorphic sections in $L^k$ and CR-holomorphic functions on $X$}
Since $(L,h)$ is a positive Hermitian line bundle,  $X$ is a strictly pseudoconvex CR manifold. The {\it  CR structure} on $X$ is defined as
follows:
 The kernel of $\alpha$ defines a horizontal hyperplane bundle $HX \subset
TX$, invariant under $J$ since $\ker \alpha = \ker d \rho \cap \ker d^c\rho$. Thus there is a splitting of the complexification of $HX$ as $HX_\C = HX \ot_\R \C =HX^{1,0} \oplus HX^{0,1}$ compatible with the splitting of $TM_\C=TM^{1,0} \oplus TM^{0,1}$.
We  define the almost-CR $\bar{\partial}_b$
operator by  $\bar{\partial}_b  = d|_{H^{0,1}}$. More concretely, if $z_1, \cdots, z_m$ are complex local coordinate on $M$, and $\dbar$ on $M$ is given by $\dbar = \sum_{j=1}^m d \bar z_j \ot \d_{\bar z_j}$, then the $\dbar_b$ operator on $X$ is given by $\dbar_b = \sum_{j=1}^m \pi^*d \bar z_j \ot \d^h_{\bar z_j}$, where $\d^h_{\bar z_j}$ is the horizontal lift of $\pa_{\bar z_j}$ from $TM^{0,1}$ to $HX^{0,1}$. Similarly one can define $\d_b$. A function $f: X \to \C$ is CR-holomorphic, if $\dbar_b f = 0$. 

A smooth section $s_k$ of $L^k$ determines a smooth function $\h s_k$ on $X$ by
\[ \h s_k(x) := \la x^{\ot k}, s_k\ra, \quad x \in X \subset L^*. \]
Furthermore $\h s_k$ is of degree $k$ under the canonical $S^1$ action $r_\theta$ on $X$, $\h s_k(r_\theta x) = e^{i k \theta} \h s_k(x)$. We denote the space of smooth section of degree $k$ by $C^\infty(X)_k$. If $s_k$ is holomorphic, then $\h s_k$ is CR-holomorphic.

We equip $X$ and $M$ with volume forms
\[ d \Vol_X =  \frac{\alpha}{2\pi}\wedge\frac{(d\alpha)^m}{m!}, \quad d \Vol_M = \frac{\omega^m}{m!},  \]
such that the push-forward measure of $X$ equals that of $M$. Then, given two smooth section $s_1, s_2$ of $L^k$, we may define the inner product
\[ \la s_1, s_2 \ra := \int_M h^k(s_1(z), s_2(z)) d \Vol_M(z). \]
Similarly, given two smooth functions $f_1, f_2$ on $X$, we may define
\[ \la f_1, f_2 \ra := \int_X f_1(x) \wb{f_2(x)} d \Vol_X(x). \]
Let $L^2(M, L^k)$ and $L^2_k(X)$ be the Hilbert spaces of $L^2$-integral sections.Then sending $s_k$ to $\h s_k$ is an isomorphism of Hilbert spaces: $L^2(M, L^k) \isoto L^2_k(X)$. Moreover,  let $\hcal^2_k(X) \subset L^2_k(X)$ be the subspace of CR-holomorphic function, then the isomorphism restricts to an isomorphism between the holomorphic sections in $L^k$ and the CR-holomorphic functions of degree $k$: $H^0(M, L^k) \isoto \hcal^2_k(X)$. 

%
%\bex
%Consider the holomorphic function $s(z)$ viewed as a section of $L^k$ of the trivial line bundle $L$ over $\C$, then 
%\[ \h s(z, \theta) = e^{i k \theta - k \varphi(z)/2} s(z) = e^{i k \theta - k |z|^2/2} s(z).\]
%The horizontal lift of $\pa_{\bar z}$ is $ \pa_{\bar z}^h =\pa_{\bar z} -  \frac{i}{2} z \pa_\theta. $ Indeed
%\[ \la \alpha,  \pa_{\bar z}^h \ra = \la d\theta - \frac{i}{2} (\bar z dz - z d \bar z), \pa_{\bar z} -  \frac{i}{2} z \pa_\theta \ra = -  \frac{i}{2} z  +  \frac{i}{2} z = 0. \]
%One may verify that $\h s$ is CR-holomorphic, i.e. $\pa_{\bar z}^h \h s = 0$: 
%\[ \pa_{\bar z}^h (\h s(z, \theta)) =(\pa_{\bar z} - i \frac{z}{2} \pa_\theta ) [e^{i k \theta - k |z|^2/2} s(z)]  = 0. \]
%%If we take $H = |z|^2$, then $dH = 2(x dx + y dy)$, $\xi_H = -x \pa_y + y \pa_x$ and $\xi_H^h = -x \pa_y^h + y \pa_x^h = y \pa_x - x \pa_y + (x^2+y^2) \pa_\theta$ and $\h \xi_H = \xi_H^h - H \pa_\theta =  y \pa_x - x \pa_y$. If $s(z) = z^m$ is a section in $L$ over $\C$, then 
%%\[ \nabla_{\xi_H} s(z) = \la m z^{m-1} dz - \bar z z^m dz, \xi_H \ra = (m - |z|^2) (-iz^m). \]
%%Then $((i/k)\nabla_{\xi_H} + H) z^m = m z^m$, for $k=1$. On the other hand $(i/k)\h \xi_H e^{i \theta - |z|^2/2} z^m =  e^{i \theta - |z|^2/2} m z^m$. This confirms the relation \eqref{e:derX}. 
%\eex

\ss{\Szego Projection kernel on $X$}
On the circle bundle $X$ over $M$, we define the orthogonal projection from $L^2(X)$ to the CR-holomorphic subspace $\hcal^2 (X) = \h \oplus_{k \geq 0} \hcal^2_k(X)$, and degree-$k$ subspace $\hcal_k^2(X)$: 
\[ \h \Pi: L^2(X) \to \hcal^2(X), \quad \h \Pi_k: L^2(X) \to \hcal_k^2(X), \quad \hPi = \sum_{k \geq 0} \hPi_k. \]
The Schwarz kernels $\hPi_k(x,y)$ of $\hPi_k$ is called the degree-$k$ \Szego kernel, i.e. 
\[ (\hPi_k F)(x) = \int_X \hPi_k(x,y) F(y) d \Vol_X(y), \quad \forall F \in L^2(X). \]
If we have an orthonormal basis $\{\h s_{k,j}\}_j$ of $\hcal^2_k(X)$, then
$$ \hPi_k(x,y) = \sum_j \h s_{k,j}(x) \wb{ \h s_{k,j}(y)}.$$ Similarly, one can define the full \Szego kernel $\hPi(x,y)$. 
The degree-$k$ kernel can be extracted as the Fourier coefficient of $\hPi(x,y)$
\[ \hPi_k(x,y) = \frac{1}{2\pi} \int_0^{2\pi} \hPi(r_\theta x, y) e^{-i k \theta} d \theta. \]

We denote the value of the kernels on the diagonal as $\hPi(x) = \hPi(x,x)$ and $\hPi_k(x) = \hPi_k(x,x)$. Then the Bergman density on $M$ can be obtained from the \Szego kernel as $ \Pi_k(z) = \hPi_k(\h z)$. 

\ss{Boutet de Monvel-Sj\"ostrand parametrix for the \Szego kernel}

Near the diagonal in $X \times X$, there exists a parametrix due to  Boutet de Monvel-Sj\"ostrand 
\cite{BSj} for the \Szego kernel of the form,  
\begin{equation} \label{SZEGOPIintroa}  
\hat{\Pi}(x,y) =  \int_{\R^+} e^{\sigma \h \psi(x,y)} s(x, y ,\sigma) d \sigma  + \hat{R}(x,y). 
\end{equation} 
where $\h \psi(x,y)$ is the almost-CR-analytic extension of $\h \psi(x,x)=-\rho(x) = \log \|x\|^2$. 
In local coordinate, let $x =  e^{i \theta_x} \frac{e_L^*(z)}{\|e_L^*(z)\|}, y =   e^{i \theta_y} \frac{e_L^*(w)}{\|e_L^*(w)\|}$,  we have then
\[ \h \psi(x,y) = i (\theta_x - \theta_y) + \psi(z, w) - \half \varphi(z) - \half \varphi(w),\]
where $\psi(z,w)$ is the almost analytic extension of $\varphi(z)$.

\ss{Bargmann-Fock Model} \label{BFmodel1}
Here we consider the trivial line bundle $L$ over $\C^m$, both as a first example to illustrate the various definitions and normalization convention and as a local model for a general \Kahler manifold. 

We fix a non-vanishing holomorphic section $e_L$ of $L$, and choose the Hermitian metric  $h$ on $L$, such that $\varphi(z) = -\log \|e_L\|^2_h(z) = |z|^2$. The \Kahler form $\omega$ is then 
\[ \omega = i \ddbar \varphi(z) = i \sum_j dz_j \wedge d \bar z_j. \]
The unit circle bundle $X$ in the dual line bundle $L^* \cong \C^m \times \C$ is given by 
\[ X = \{ (z, \lambda) \in  \C^m \times \C \mid \| \lambda e_L^*\| =1 \} = \{ (z, \lambda) \in  \C^m \times \C \mid | \lambda|  = e^{- \half|z|^2} \}.  \]
We may then choose a trivialization of $X \cong \C^m \times S^1$, with coordinate $(z, \theta)$, 
\[ (z, \theta) \mapsto e^{i\theta} e_L^*(z) / \|e_L^*(z) \| = e^{i \theta - \half |z|^2} e_L^*(z) \in L^*. \]

The contact form $\alpha$ on $X$ is then 
\[ \alpha = d\theta - \frac{i}{2} \sum_j (\bar z_j dz_j - z_j d \bar z_j). \]

If $s(z)$ is a holomorphic function (section of $L^k$) on $\C^m$, then its CR-holomorphic lift to $X$ is 
\[ \h s(z, \theta) = e^{k(i \theta - \half |z|^2)} s(z). \]
Indeed, the horizontal lift of $\pa_{\bar z_j}$ is $ \pa_{\bar z_j}^h =\pa_{\bar z_j} -  \frac{i}{2} z_j \pa_\theta, $
and $\pa_{\bar z_j}^h \h s(z, \theta) = 0$. 

The \Szego kernel $\h \Pi(\h z, \h w)$ for $X=\C^m \times S^1$ is given by 
\[  \h \Pi(\h z, \h w) = \sum_{k>0} \kk^m e^{i k (\theta_z - \theta_w) + k(z \bar w - \half |z|^2 - \half |w|^2)} =\sum_{k>0} \kk^m e^{i k (\theta_z - \theta_w + \Im (z \bar w)) - \half k |z-w|^2} , \] where $\h z =(z, \theta_z), \; \h w = (w, \theta_w),$ and 
where the $k$-th summand is $\h \Pi_k(\h z, \h w)$. 

The Bergman kernel $\Pi_k(z, w)$ for $M=\C^m$ is given by
\[ \Pi_k(z, w) = \kk^{m} e^{k z \bar w}. \]
The Bergman density is the norm contraction of $\Pi_k(z,w)$ on the diagonal
\[ \Pi_k(z) = \Pi_k(z,z) \|e_L^{\ot k} \|^2 = \kk^{m} e^{k z \bar z} e^{-k|z|^2} = \kk^m. \]

In general, for an $m$-dimensional complex vector space $(V, J)$ with a constant \Kahler form $\omega$, we may define a `ground state'
\[ \Omega_{\omega, J} (v) = e^{-\half \omega(v, J v)}, \]
as a real valued function on $V$. 
The Bargmann-Fock Hilbert space for $(V, \omega, J)$ is 
\[ \hcal_{\omega, J} = \{ \psi \Omega_{\omega, J}: \text{ $\psi$ is $J$-holomorphic, and $\psi \Omega_{\omega, J} \in L^2 (V, dL)$}\} \]
where $L$ is the standard Lebesgue measure on $V$.

\ss{Osculating Bargmann-Fock model and Near-Diagonal Scaling Asymptotics}
At
each $z \in M$ there is an osculating Bargmann-Fock or Heisenberg model,
defined as above for the data $(T_z M, J_z, \omega_z)$.
We denote the model Heisenberg \szego kernel on the tangent space by 
\begin{equation} \label{TANGENT}  \hat{\Pi}^{T_z M}_{\omega_z, J_z}: L^2(T_z M \times S^1) \to \hcal(T_z M, J_z, \omega_z)
= \hcal_J.
\end{equation}
If we choose linear coordinates $(z_1, \cdots, z_m)$ on $T_z M$ such that $(T_z M, \omega_z, J_z) \cong (\C^m, \omega_{std}, J_{std})$ , and the obvious coordinate $\theta$ on $S^1 = \R / 2\pi \Z$, then we have
\begin{equation} \label{SCALINGB} \h\Pi^{T_z M}_{\omega_z, J_z}(u, \theta_1,  v, \theta_2)  = (2\pi)^{-m}   e^{i (\theta_1 - \theta_2)} e^{u\cdot \bar{v} - \half (|u|^2 + |v|^2)}=    e^{i (\theta_1 - \theta_2)} e^{i \Im u\cdot \bar{v} - \half |u -v|^2}
\end{equation}

The following near-diagonal asymptotics of the \szego kernel
is the
 key analytical result on which our analysis of the scaling limit for
correlations of zeros  is based.
The lifted \szego kernel is shown in \cite{ShZ02}, and in Theorems 2.2 - 2.3 of \cite{LuSh15} to have the scaling asymptotics. Here we only state the first version since it will be enough for our purpose. 

\begin{theo}[\cite{ShZ02}]
Let $(L, h) \to (M, \omega)$ be a positive line bundle over an $m$-dimensional compact \kahler manifold with \kahler form $\omega = i F_\nabla$. Let $e_L$ be a holomorphic local frame for $L$ and $z_1, \cdots, z_m$ be complex coordinates about a point $z_0 \in M$ such that the \kahler potential $\varphi = -\log \|e_L(z)\|^2_h = |z|^2 + O(|z|^3)$. Then for a $k \in \Z^+$, 
\[  \Pi_k \left( \frac{u}{\sqrt k}, \frac{\theta_1}{k};  \frac{v}{\sqrt k}, \frac{\theta_2}{k} \right) = k^m \h
\Pi_{\omega_z, J_z}^{T_z M}(u, \theta_1; v, \theta_2) (1 + \sum_{r=1}^N k^{-r/2} b_r(u, v) + k^{-(N+1)/2} E_{kN}). \]
where 
\begin{itemize}
\item each $b_r(u,v)$ is a polynomial (in $u, v, \bar u, \bar v)$ of degree at most $5r$. 
\item for all $a > 0$ and $j\geq 0$, there exists a positive constant $C_{jka}$ such that 
\[ | D^j E_{kN}(u,v) | \leq C_{jNa}, \text{ for } |u|+|v| < a. \]
\end{itemize}
\end{theo}

\section{\label{THAFSECT} Proof of Theorem \ref{AF}}

This is the simplest of the results because it does not involve dynamics
of the Hamiltonian flow. That is, it only involves the unitary group \eqref{UNSCALED} of `pseudo-differential' Toeplitz operators and not
the group $U_k(t)$ of Fourier integral Toeplitz operators.

  We  first prove a smoothed version of Theorem \ref{AF}. 
We recall from \eqref{fPBK} that a  smoothly weighted Bergman density is defined by
\begin{equation} \label{fPBK2} \Pi_{k,f} (z) = \Pi_k f({H}_k) \Pi_k (z)
= \sum_{\mu_{k,j}}
f (\mu_{k,j}) \Pi_{k,j}(z) =  \sum_{\mu_{k,j}}
f (\mu_{k,j}) \| s_{k,j}(z)\|_{h^k}^2, 
 \end{equation}
 where $f \in \scal(\R)$ and the sum is over the eigenvalues $\mu_{k,j}$ of the operator $H_k$ defined in \eqref{TOEP}. Note that since $H_k$ commutes with $\Pi_k$ and $\Pi_k \circ \Pi_k = \Pi_k$, we could have written $f(H_k) \Pi_k$ instead of $\Pi_k f(H_k) \Pi_k$.
 The sharp partial Bergman kernels $\Pi_{k, E}$ morally corresponds to the case $f = {\bf 1}_{(E_{\min}, E)}$, where $E_{min} = \min_{z \in M}\{ H(z) \}$.

\begin{prop}\label{AFSMOOTH}  Let $\omega$ be a $C^{\infty}$ metric  on $M$and let $H \in C^{\infty}(M)$. Let
$f \in C_c^{\infty}(\R)$.
Then the  density of states of the smoothly weighted Bergman kernel is given by the asymptotic
formulae:
$$   \Pi_{k, f}(z)  \simeq \Pi_k(z)(
    f(H(z)) + c_{1,f}(z) k^{-1} + c_{2,f}(z)k^{-2} + \cdots),    
 $$
 where  $c_{i,f}(z)$ depend on $f$ upto its $i$-th derivative at $H(z)$. 
\end{prop}
We give two proofs for the above proposition, one using Helffer-\Sjostrand formula and the fact $\Pi_k (\lambda - H_k)^{-1} \Pi_k$ is a Toeplitz operator, the other using Fourier transformation of $f$ and the fact that $\Pi_k e^{i t H_k} \Pi_k$ is a Toeplitz operator.

\subsection{Proof of Proposition \ref{AFSMOOTH} using Helffer-\Sjostrand formula}

\begin{proof}

We use  the Helffer-Sjoestrand formula \cite{DSj}. Let $f \in C_c^{\infty}(\R)$ 
and let    $\wt{f} (\lambda) \in C_c^{\infty}(\C)$ be an almost analytic extension
of $f$ to $\C$, with $\dbar \wt{f} = 0$ to infinite order on $\R$. Then
$$f({H}_k) = -\frac{1}{ \pi} \int_{\C} \dbar \wt{f}(\lambda) (\lambda - {H}_k)^{-1} dL(\lambda) $$
where $dL$ is the Lebesgue measure on $\C$.  We recall that
the almost analytic extension is defined as follows: Let $\psi(x) = 1$ on 
$\rm{supp} f$ and let $\chi$ be a standard cutoff function. Then, define
\begin{equation} \label{AA} \wt{f}(x + i y) = \frac{\psi(x)}{2 \pi} \int_{\R} e^{i (x + i y) \xi} \chi(y \xi) \hat{f}(\xi) d \xi. \end{equation}
It is verified in \cite[p. 94]{DSj} that this defines an almost analytic extension.  This formula has previously been adapted to
Toeplitz operator in \cite{Ch03}.

 For $\Im \lambda \not= 0$ there exists a unique
semi-classical symbol $\sum_{j=0}^\infty k^{-j} b_j(w; \lambda)$ with $b_0 = \frac{1}{\lambda - H(w)}$ and a Toeplitz
operator $\hat{B}_k(\lambda)$ as an approximation for $\Pi_k (\lambda - H_k)^{-1} \Pi_k$,  such that
$$\hat{B}_k(\lambda) \circ (\lambda- {H}_k) = \Pi_{k} + R_{B, k}(\lambda), \;\; \lambda \in \C \RM \R,$$
where  $R_{B, k}(\lambda)$ is a residual Toeplitz operator (i.e. of order $k^{-\infty}$).  
Thus
$$ \Pi_{k,f} := \Pi_k f({H}_k)  \Pi_k \approx  -\frac{1}{ \pi} \int_{\C}  \dbar \wt f(\lambda) \h B_k(\lambda)  dL(\lambda) $$ 
is a Toeplitz operator with complete symbol
$$\sigma_{k,f}(w) \sim - \sum_{j=0}^{\infty} \frac{k^{-j} }{ \pi} \int_{\C} \dbar \wt{f}(\lambda) b_j(w; \lambda) dL(\lambda), $$
and
principal
symbol $$ \sigma_f^{prin} (w) =- \frac{1}{\pi} \int_{\C} \dbar \wt{f}(\lambda) (\lambda - H(w))^{-1}dL(\lambda) = f(H(w)). $$

We can then express the density for $\Pi_{k,f}$ as
$$\Pi_{k, f}(z) = \int_X \h \Pi_k(\h z,\h w) \sigma_{k, f}(w) \h \Pi_k(\h w,\h z) d\Vol_X(w) $$
and use the Boutet de Monvel-\Sjostrand parametrix to compute the
resulting expansion stated in Proposition \ref{AFSMOOTH}.

Finally, we show how the coefficients in the expansion $c_{i,f}(w)$ depends on   $j$-th derivative of $f$ at $H(w)$, for $j \leq i$. For this it suffices to prove the same statements for the terms,
$$  \int_{\C} \dbar \wt{f}(\lambda) b_j(w; \lambda) dL(\lambda)$$
for $j \geq 1$. As discussed in \cite{DSj}, the lower order terms $b_j(w; \lambda)$ of $\hat{B}_k$ have the form, $$b_j(w; \lambda) = P_j(w; \lambda) (\lambda - H(w))^{-j -1}, $$
where $P_j$ is  polynomial in $\lambda$. Writing
$ (\lambda - H(w))^{-j -1}  = \frac{1}{j!} \frac{\partial^j}{\partial \lambda^j}  (\lambda - H(w))^{-1}$ 
and integrating by parts gives the value as a sum of holomorphic derivatives 
$\dbar (\d^{\ell}\wt{f}  (\lambda)) $ evaluted at $\lambda  = H(w)$ of degree $\ell \leq j$. This shows $c_{i,f}(w)$ depends on $f$ through $f^{(0)}(H(z)), \cdots, f^{(i)}(H(z))$. 
\end{proof}

We get immediately the following corollary
\bc \label{vanishing}
If $f$ vanishes in an open neighborhood of $H(z)$, then 
\[ \Pi_{k,f}(z) = O(k^{-\infty}). \]
\ec
%\bpf
%By the previous proposition, 
%\[ k^{-m} \Pi_{k,f}(z) = f(H(z)) + c_{1,f}(z) k^{-1} + \cdots \]
%where $c_{i,f}(z)$ is a linear combination of $f^{(0)}, \cdots, f^{(i)}$ evaluated at $H(z)$. Since $f$ vanishes near $H(z)$,  we have$c_{i,f}(z)=0$ for all $i \geq 0$. 
%\epf

\subsection{Proof of Proposition \ref{AFSMOOTH} using Fourier Transformation}

Since we use the Fourier inversion formula in other applications of the 
functional calculus, we digress to describe an alternative to  the Helffer-\Sjostrand formula above.  

We may construct $W_k(t) := \Pi_k e^{i t {H}_k} \Pi_k$ as a Toeplitz operator
of the form
\begin{equation} \Pi_k u_k(w; t) \Pi_k,\;\; u_k(w; t) = e^{i t H(w)} + \sum_{j=1}^{\infty} k^{-j} u_j(w; t), \end{equation}
where $u_k(w; t)$ denotes multiplication by the function $u_k(w; t)$, where $w \in M, t \in \R$. For instance,
we can write the equation for $W_k(t)$ as
 the unique solution of the propagator initial
value problem,

\begin{equation} \label{SCHEQ}  \left(\frac{1}{i } \frac{d}{dt}  - {H}_{k} \right) W_k(t) = 0 , \;\;\; \text{with  }  W_k(0)  =\Pi_k. 
\end{equation}
%
%\begin{prop}\label{PARAV} There exists a semi-classical  symbol $u_{k}(t)$, depending on parameter $t \in \R$,  so that $$u_{k}(w; t) =e^{i t H(w)} + \sum_{j \geq 1} u_j(w; t) k^{-j} $$ and
%$$\left\{ \begin{array}{l}\left(\frac{1}{i } \frac{d}{dt}  - {H}_{k} \right) 
%W_k(t) : = R_k(t), \\ \\\;\;\;\;
%W_k(0) = \Pi_k. \end{array} \right.
%$$
%where $R_k(t)$ is a residual (neglible) analytic Toeplitz operator. \end{prop}
There exists a semi-classical symbol 
$$\sigma_{k,f} (w) = f(H(w)) + k^{-1} c_{f,1}(z) + \cdots, $$ 
so that
\begin{equation} \label{f(H)} 
\Pi_k f( H_k) \Pi_k= \Pi_k \sigma_{k,f} (w) \Pi_k, \;\; \;  \end{equation}
Indeed,
\[ \Pi_{k,f} =  \Pi_k f(H_k) \Pi_k
=  \int_{\R} \hat{f}(t) \Pi_k e^{i t \h H_k} \Pi_k(z,z) \frac{dt}{2\pi}
=  \Pi_k \left( \int_{\R}  \hat{f}(t)  u_k(w; t) \frac{dt}{2\pi} \right) \Pi_k
\]
Hence, the symbol for $\Pi_{k,f}$ is then $\sigma_{k,f} (w) = \int_{\R}  \hat{f}(t)  u_k(w; t) \frac{dt}{2\pi}$, with the principal symbol $f(H)$. 

\subsection{Proof of Theorem \ref{AF}}
We first prove the result in Theorem \ref{AF} about $\Pi_{k}(z)^{-1} \Pi_{k,E}(z)$, Eq. \eqref{TOEPeq2}. It suffices to prove the result in the case where $z$ lies in the forbidden region $H(z) > E$. Indeed, since for $H(z) <E$, we may use $-H$ and $-E$ as the energy function and the threshold, write 
\[ \Pi_{k}(z)^{-1} \Pi_{k, H < E}(z) = 1 - \Pi_{k}(z)^{-1} \Pi_{k, -H < -E}(z), \]
and use $\Pi_{k, -H < -E}(z) = O(k^{-\infty})$ from the result in the forbidden region. Hence from now on, we  assume $H(z) > E$. 

Let $\epsilon$ be small enough positive number, such that $H(z) > E + \epsilon$. It suffices to assume $f \in C_0^{\infty}(\R)$, $f \geq 0$,  with  $f \equiv 1$ on the interval $(E_{min}, E)$ and $f \equiv 0$ outside of $(E_{min} - \epsilon, E+\epsilon)$,  since the general result can be obtained by taking the difference of two such functions.
Then, $f(x) > {\bf 1}_{[E_{min}, E]}(x)$, and
\[ \Pi_{k,f}(z) = \sum_{j} f(\mu_{k,j}) \|s_{k,j}(z)\|^2 \geq  \sum_{\mu_{k,j} < E}  \|s_{k,j}(z)\|^2 = \Pi_{k,E}(z). \]
Since we have $f(H(w)) \equiv 0$ for $w$ in an open neighborhood of $z$, by Corollary  \ref{vanishing}, 
$ \Pi_{k,f}(z) = O(k^{-\infty})$. Since $\Pi_{k,E}(z) > 0$, we have $$|\Pi_{k,E}(z)| = \Pi_{k,E}(z) <  \Pi_{k,f}(z) = O(k^{-\infty}).$$ 

To prove the statement in Eq. \eqref{TOEPeq1}, we split the sum in $\Pi_{k,f}(z)$ according to the eigenvalue
\[ \Pi_{k,f}(z) = \Pi_{k,f, <E}(z) + \Pi_{k,f, >E}(z), \]
where 
\[ \Pi_{k,f, <E}(z) := \sum_{\mu_{k,j} < E} f(\mu_{k,j}) \|s_{k,j}(z)\|^2, \quad \Pi_{k,f, >E}(z) := \sum_{\mu_{k,j} > E} f(\mu_{k,j}) \|s_{k,j}(z)\|^2. \] 

For  $z \in \fcal$, $H(z)>E$, we need to prove that 
\[ 
\lim_{k \to \infty} k^{-m} \Pi_{k,f, <E}(z) = 0. 
\]
We may define a non-negative smooth function $\chi$ on $\R$ that is $1$ for  $x < E$, and $0$ for $x > H(z) > E$. Then
\[ 0 < \Pi_{k, f, <E} (z) < \Pi_{k, f  \chi}(z) = O(k^{-\infty}), \]
where we used Corollary  \ref{vanishing} and that $f\chi(w)$ vanishes identically for $w$ in a neighborhood of $H(z)$. 
Similarly, if $H(z) < E$, then 
\[ 
\lim_{k \to \infty} k^{-m} \Pi_{k,f, >E}(z) = 0, 
\]
and we get
\[ \lim_{k \to \infty} k^{-m} \Pi_{k,f, <E}(z) = \lim_{k \to \infty} k^{-m} \Pi_{k,f}(z) = f(H(z)). \]
This finishes the proof of  the statement of Theorem \ref{AF} regarding the leading order asymptotic behavior of $\Pi_{k, f}(z)$.
%
%{\red
%\subsection{\label{PROBS1} Asymptotics for the family $\mu_k^z$}
%
%\edit{Statement of Theorem 1 has result about the family of measure, maybe we should remove this subsection.}
% 
%  In this section we consider the family of measures $\mu_k^z(\lambda)$ of \eqref{mukzdef} (i).
%The Fourier transform of $\mu_k^z(\lambda)$ is given by
%\begin{equation} \label{FT} [\fcal_{\lambda \to t} d \mu_k^z](t)
%= \int_{\lambda} e^{- i t \lambda} d \mu_k^z(\lambda)
%=  \sum_{j}
% \Pi_{k,j}(z, z) e^{- it\mu_{k,j}} . 
% \end{equation}
%  
% Thus, for $f \in \scal(\R)$  \eqref{fPBK}  is given by \begin{equation} \label{fPBK3} \Pi_k f(\h H_k) \Pi_k
%= 
%\int_{\R} f(\lambda) 
% d\mu_k^z(\lambda) = \int_{\R} \hat{f}(t) [\fcal_{\lambda \to t} d\mu_k^z](t).\end{equation}
% 
% We may thus re-state Proposition \ref{AFSMOOTH} in terms
% of the measures $d\mu_k^z$ as follows:
% 
% \begin{prop}\label{muk} If $f \in C_c^{\infty}(\R)$, then
% $$\int_{\R} f(\lambda) 
% d\mu_k^z(\lambda) \sim  \bcs
%  f(H(z)) + c_1k^{-1} + c_2k^{-2} + \cdots,    & \mbox{for } z \in \acal = 
%H^{-1}(P)
%\\[10pt]  
%  O(k^{-\infty}), \;\; & \text{ for } z \in \fcal
%\ecs $$
%\end{prop}
%}

\section{\label{LIFTFLOW} Hamiltonian flow and its contact lifts to $X_h$}

In order to prove the main results, we need to quantize the Hamiltonian
flow of a Hamiltonian $H: M \to \R$. The definition is based on lifting
the Hamiltonian flow to a contact flow on $X_h$. The purpose of this
section is to explain this lift.

\subsection{Lifting the Hamiltonian flow to a contact flow on $X_h$.}\label{LIFT}
Let $H$ be a Hamiltonian function on $(M, \omega)$. Let $\xi_H$ be the Hamiltonian vector field associated to $H$, that is, 
\[ d H(Y) = \omega(\xi_H, Y) \]
for all  vector field $Y$ on $M$. Let $g^t$ be the flow generated by $\xi_H$. 

Recall $(X,\alpha)$ is a contact manifold, and $X \to M$ is a circle bundle, such that $d\alpha = \pi^* \omega$. 
We abuse notation and still use $H$ for the lifted Hamiltonian function $\pi^* H$ on $X$. The horizontal lift $\xi_H^h$ of $\xi_H$ is given by the following conditions
\[ \alpha(\xi_H^h) = 0, \quad \pi_* \xi_H^h = \xi_H. \]
The contact lift $\h \xi_H$ of $\xi_H$ is then defined by
\be \hat{\xi}_H := {\xi}_H^h -  H R.  \label{hxi} \ee

 We define $\hat{g}^t: X_h \to X_h$ to be the flow generated by $\h \xi_H$, 
\begin{equation} \label{gtdef} \h{g}^t= \exp t  \h \xi_{H}.  \end{equation}
\bl
The flow   $\h g^t$  preserves the contact form $\alpha$, and  commutes with the $S^1$ action of rotation
in the fibers of $X \to M$.
\el
\bpf
Since $d \alpha = \omega$, we have
\[\lcal_{\hat{\xi}_H}\alpha = \lcal_{{\xi}_H^h -  H R} \alpha = (\iota_{{\xi}_H^h  -  H R} \circ d + d \circ \iota_{{\xi}_H^h - H R}) \alpha = \iota_{\xi_H^h} \pi^* \omega  + d(- H \alpha(R)) = H  - H = 0. \]
Since $\h \xi_H$ preserves $\alpha$, hence $\lcal_{\h \xi_H} (R) = 0$, i.e. $[\h \xi_H, \pa_\theta]=0$, and the flow $\h g^t$ commutes with the fiberwise rotation.
\epf

Let $(z, \theta)$ be a local coordinate for $X|_U \cong U \times S^1$ over open neighborhood $U$,  as in \eqref{Xcor}. 
\bl
In coordinate $(z, \theta)$, $\xi_H^h$ and $\h \xi_H$ can be written as 
\[\xi_H^h = (\xi_H, \half \la d^c \varphi, \xi_H \ra \pa_\theta), \quad \h \xi_H = (\xi_H, (\half \la d^c \varphi, \xi_H \ra - H)\pa_\theta). \]
And the  flow $\h g^t$ has the form
$$\begin{array}{lll} \hat{g}^t(z, \theta)  & =  & (g^t(z), \theta 
  +   \int_0^t \half \la d^c \varphi, \xi_H \ra(g^s(z))ds 
- \int_0^t H (g^s(z)) ds)\\ && 
\\ & =  & (g^t(z), \theta 
  +   \int_0^t \half \la d^c \varphi, \xi_H \ra(g^s(z))ds  - t H(z).   \end{array}$$
\el
\begin{proof} 
Since $\alpha = d\theta - \half \pi^* d^c \varphi$, and $\la \alpha, \xi_H^h \ra = 0$, we have then 
\[ \la d \theta, \xi_H^h \ra = \la \half \pi^* d^c \varphi, \xi_H^h \ra = \la \half d^c \varphi, \xi_H \ra. \]
Hence $\xi_H^h = (\xi_H, \half d^c \varphi(\xi_H) \pa_\theta)$. The formula for $\h \xi_H$ follows from \eqref{hxi}. The statement for $\h g^t$ follows since $H$ is constant along the Hamiltonian flow.
\end{proof}

Since $\hat{g}^t$ preserves $\alpha$ it preserves the horizontal distribution $H(X_h) = \ker \alpha$, i.e.
\begin{equation} \label{HSPLIT} D \hat{g}^t: H(X)_x \to H(X)_{\hat{g}^t(x)}. \end{equation} It also preserves the vertical
(fiber) direction and therefore preserves the splitting $V \oplus H$ of $T X$.  When $g^t$ is non-holomorphic, $\hat{g}^t$ is not CR holomorphic, i.e. does not preserve the horizontal complex structure $J$ on $H(X)$ or the
splitting of $H(X) \otimes \C$ into its $\pm i $ eigenspaces.

\subsection{Linear Hamiltonian function and Heisenberg group action} \label{BFmodel2}

Let $\C^m$ be equipped with coordinates $z_j = x_j + i y_j$ with $j=1, \cdots, m$. We use $x_i, y_i$ as coordinate on $\R^{2m}$. Let $L$ be the trivial bundle over $\C^m$, and $\varphi(z) = |z|^2$. 

A linear Hamiltonian function $H$ on $\C^m$ has the form,
\be H(x,y) = \Re (\alpha \cdot \bar z) = \half (\alpha \bar z + \bar \alpha z),\ee
for some $0 \neq \alpha \in \C^m$. Then 
\bea
\xi_H &=& \sum_j \frac{1}{2i} (\alpha_j \pa_{z_j} -  \bar \alpha_j \pa_{ \bar z_j}),  \\
\xi_H^h &=& \sum_j \frac{1}{2i} (\alpha_j (\pa_{z_j} + \frac{i}{2} \bar z_j \pa_\theta)  -  \bar \alpha_j (\pa_{ \bar z_j }- \frac{i}{2}   z_j \pa_\theta)) = \xi_H + \half H \pa_\theta,  \\
\hat{\xi}_H &=& \xi_H^h - H \pa_\theta =  \xi_H - \half H \pa_\theta.  \\
\eea
where we abused notation and write $(\xi_H, 0)$ on $U \times S^1$ as $\xi_H$. 

The lifted Hamiltonian flow   $\h g^t(z) = \exp( t \h \xi_H)$ is then
\begin{equation} \label{LINEARIZATION}  \h g^t(z, \theta) = (z + \frac{\alpha t}{2i}, \theta - \frac{t}{4}  (\alpha \bar z + \bar \alpha z))
\end{equation}

%
%\ss{Heisenberg group and Bargmann-Fock Representation}
%To leading order the asymptotics   in Theorem \ref{RSCOR} and Theorem \ref{ELLSMOOTH}  at $z$ can be reduced to the asymptotics of the linearized kernels in the {\it osculating  Bargmann-Fock model} at $z$. Namely, each tangent space $T_x X_h$ may be identified with the simply connected Heisenberg group, and the the complexification of the CR 
%subspace $H_x$ with $\C^m$.  In this
%section we review the definition of the Heisenberg model, the  Bargmann-Fock model and the
%osculating Bargmann-Fock models at points of a \kahler manifold. 

The linear Hamiltonian function generate translation on $\C^m \times S^1$, which is exactly the action of the reduced Heisenberg group, which we now review.
The simply connected Heisenberg group \cite{F89, S93} of dimension $(2m + 1)$ is $\H^m \cong \R^{2m} \times \R = \{(x, y, t) \mid x, y \in \R^n, t \in \R\}$, with a multiplication law
\be (x, y, t) (x', y', t') = (x+ x', y+y', t  + t' -  (x y' - y x')). \label{HeisAct} \ee
The reduced Heisenberg group is $\H^m_{red} = \R^{2m} \times (\R / 2\pi \Z)$. The center of $\H^m_{red}$ is the circle subgroup $S^1 \cong \{ (0,0, \theta) \}$. If we identify $\R^{2m}$ with $\C^m$ by $z_i = x_i + \sqrt{-1}  y_i$, and $\R / 2\pi \Z$ with the argument of $\C$, then we may identify $\H^m_{red} \cong \C^m \times S^1$.
And the group action is given by 
\[ (z, \theta) \circ (z', \theta') = (z+z', \theta+\theta' + \Im( z \bar z')). \] Furthermore, $\H^m_{red}$ can be identified  with the unit co-circle bundle $X$ of the dual bundle  $L^*$ of the positive  hermitian line bundle over $\C^m$ with \kahler potential $\varphi(z) = |z|^2$.  

%As explained in \S \ref{ss:circle}, the circle bundle $X$ is a contact manifold with contact form
%\[ \alpha = d \theta + \frac{i}{2} \sum_j (z_j d \bar z_j - \bar z_j d z_j), \]
%and the base of the circle bundle $\C^m$ is equipped with a \kahler form
%\[ \omega = i \sum_j d z_j \wedge d \bar z_j. \]

\bl
The contact form $\alpha = d\theta + \frac{i}{2} \sum_j (z_j d\bar z_j - \bar z_j dz_j)$ on $\H^m_{red}$ is invariant under the left multiplication
\[ R_{(z_0, \theta_0)}: (z, \theta) \mapsto (z_0, \theta_0)  \circ (z, \theta) = (z + z_0, \theta + \theta_0 + \frac{ z_0 \bar z-  \bar z_0 z }{2i} ). \]
\el
\bpf
\[ R_{(z_0, \theta_0)}^* \alpha|_{(z, \theta)} = d (\theta + \theta_0 + \frac{\bar z z_0 - \bar z_0 z}{2i} ) + \frac{i}{2} \sum_j ((z_j + z_{0j}) d \bar z_j - (\bar z_j+\bar z_{0j}) d z_j) = \alpha|_{(z, \theta)}.\]
\epf

\bl
The lifted contact flow $\h g^t$ generated by $H =\half (\alpha \bar z + \bar \alpha z) $ acts on $\H^m_{red}$ by left group multiplication by $(\frac{\alpha t}{2 i}, 0) \in \H^m_{red}$. 
\el
\bpf
\[ (\frac{\alpha t}{2 i}, 0)  \circ (z, \theta) = (z + \frac{\alpha t}{2 i}, \theta + \Im (\frac{\alpha t}{2 i} \bar z)) = (z + \frac{\alpha t}{2 i}, \theta  - \half \Re (\alpha \bar z) t) = \h g^t(z, \theta). \]
\epf

The Hardy space of square-integrable CR holomorphic functions $\hcal$ on $X$ is preserved under the reduced Heisenberg group $\H^m_{red}$, and decomposes into Fourier components according to the action by the central subgroup $S^1$
\[ \hcal = \bigoplus_{n \in \Z, n>0} \hcal_n, \quad \hcal_n = \{ e^{i n (\theta - |z|^2/2)} f(z) \mid f(z) \text{ holomorphic }, \int_{\C^m} |f|^2 e^{-n|z|^2} dL < \infty \}, \]
where $dL$ is the Lesbegue measure on $\C^m$ and the integrability condition forces $n>0$. 

 In the case of the simply connected $(2m + 1)$-dimensional Heisenberg group $\H^m$,  the Szeg\"o projector
$S(x, y)$  is given by 
convolution with 
\begin{equation} \label{HEISSZEGO} K(x)  = C_m \frac{\partial}{\partial t} (t + i |\zeta|^2/2)^{-m} =
C_m   \int_0^{\infty} e^{r  (it  - |z|^2/2)} r^{m} d r  , \;\;\; x = (z, t), \end{equation}
for some constant $C_m$ depending on the dimension $m$. More precisely, if $x = (z, t)$, $y = (w, s)$, and the projector is given by
\[ (S f)(x) = \int S(x,y) f(y) dy = \int K(y^{-1}x) f(y) dy \]
and 
\be S(x,y) = K(y^{-1}x) = K(z - w, t-s -\Im(w \bar z)) = C_m \int_0^\infty e^{i r (t-s)} e^{ r(z \bar w -   |z|^2/2 -  |w|^2/2)} r^m dr 
\ee
 In the reduced Heisenberg group,
\begin{equation} \label{HEISSZEGOred} S_{red }(x, y) = \sum_{n \in \Z} K(y^{-1}x (0,2\pi n)) = \sum_{n \in\Z}  C_m \int_0^\infty e^{i r (t-s)} e^{ r(z \bar w -   |z|^2/2 -  |w|^2/2)}  r^m dr 
\end{equation}
Using the Poisson summation formula $\sum_{m \in \Z} e^{2\pi i k x} = \sum_{n \in \Z} \delta(x-n)$, we have
\[ S_{red}(x,y)= C_m \sum_{k \in \Z_{>0}}  k^m e^{i k (t-s)} e^{ k(z \bar w -   |z|^2/2 -  |w|^2/2)} , \quad \text{where } x=(z, t), y=(w,s) . \]
If we denote the Fourier components relative
to the $S^1$ action of $\H^m_{red}$ by $\hat{\Pi}_k^\H(x,y)$, then we have
\begin{equation}\label{szegoheisenbergintro}   \hat{\Pi}_k^\H(x,y)=C_m k^m
e^{i k (t-s )} e^{ k(z\cdot\bar w -\half |z|^2
-\half|w|^2) }, \end{equation}
and  $S_{red}(x,y) = \sum_{k>0} \hat{\Pi}_k^\H(x,y)$. 

% 
%{\red 
%
%\subsection{Gradient flow and its lift in the $C^{\infty}$ setting}
%\edit{We do not need its lift, do we?}
%
%Recall that  the gradient vector field $\nabla H$ is related to  the Hamiltonian vector field $\xi_H$, for any $Y \in Vect(M)$,
%\[ dH(Y) = \omega(\xi_H, Y), \quad dH(Y) = g(Y, \nabla H), \quad g(X, Y) = \omega(X, JY) \]
%hence 
%\[ \nabla H = J \xi_H.\]
%
%We denote by $$F^t: M \to M, \;\; F^t(z)  = \exp t \nabla H (z)$$ the gradient flow of $H$ on $(M, J, \omega)$.
%Although $\nabla H = J \xi_H$, the Hamilton flow and gradient flow do
%not generally commute and do not define a $\C$ action.
%
%Suppose that $\rho \in C^{\infty}(M)$. Then, just as if one $\exp t \nabla H$
%were a true exponential, 
%$$\rho(F^{\frac{t}{\sqrt{k}}}(z)) =\rho(z) + \frac{t}{\sqrt{k}} (\nabla H)(\rho)(z)
%+ \half \frac{t^2}{k}  (\nabla H)^2(\rho)(z) + O(k^{-3/2}). $$
%
%As with the Hamiltonian flow, $\nabla H$ may be horizontally lifted
%to a vector field $\xi^h_{\nabla H}$ on $X_h$ and it equals $J \xi_H^h$ where
%$J$ is the complex structure on the horizontal spaces of $X_h$ with
%respect to the connection $\alpha$. We further define
%$$V: = \xi^h_{\nabla H}  -2 \pi i H_{\tau} \frac{\partial}{\partial \theta} $$ and denote its flow by 
%by $\hat{f}^t$.
%
%}
% 

\section{Toeplitz Quantization of Contact Transformation} \label{TQD}
Let $(M, \omega, L, h)$ be a polarized \Kahler manifold, and $\pi: X \to M$ the unit circle bundle in the dual bundle $(L^*, h^*)$.  $X$ is a contact manifold, and can be equipped with a contact one-form $\alpha$, whose associated Reeb flow $R$ is the rotation $\pa_\theta$ in the fiber direction of $X$. Any Hamiltonian vector field $\xi_H$ on $M$ generated by a a smooth function $H: M \to R$ can be lifted to a contact Hamiltonian vector field $\h \xi_H$ on $X$. In this section, we review the relevant result about the quantization of this contact vector field $\h \xi_H$, acting on holomorphic sections of $H^0(M, L^k)$, or equivalently on CR-holomorphic functions on $X$: $\hcal(X) = \oplus_{k\geq 0} \hcal_k(X)$. We follow the exposition of \cite{RZ, BG81} closely. 

An operator $T: C^\infty(X) \to C^\infty(X)$ is called a {\em Toeplitz operator of order $k$}, denoted as $T \in \tcal^{k}$, if it can be written as $T = \h \Pi \circ Q \circ \h \Pi$, where $Q$ is a pseudo-diffferential operator on $X$ . Its principal symbol $\sigma(T)$ is the restriction of the principal symbol of $Q$ to the symplectic cone 
\[ \Sigma = \{(x, r \alpha(x)) \mid r > 0\} \cong X \times \R_+ \subset T^*X.\] 
The symbol satisfies the following properties
\[
\bcs
\sigma(T_1 T_2) = \sigma(T_1) \sigma(T_2); \\
\sigma([T_1, T_2]) = \{ \sigma(T_1), \sigma(T_2)\}; \\
\text{If $T \in \tcal^k$, and $\sigma(T) = 0$, then $T \in \tcal^{k-1}$.}
\ecs
\]
The choice of the pseudodifferential operator $Q$ in the definition of $T = \hPi \, Q \, \hPi$ is not unique. However, there exists some particularly nice choices. 
\bl[\cite{BG81} Proposition 2.13] \label{lm:niceQ}
Let $T$ be a Toeplitz operator on $\Sigma$ of order $p$, then there exists a pseudodifferential operator $Q$ of order $p$ on $X$, such that $[Q, \hPi] = 0$ and $T = \hPi \, Q \, \hPi$. 
\el

Now we specialize to the setup here, following closely \cite{RS}. Consider an order one self-adjoint Toeplitz operator 
\[ T = \h \Pi \circ (H \cdot \mathbf{D}) \circ \h \Pi, \] where $\mathbf{D} = (-i \pa_\theta)$ and $\pa_\theta$ is the fiberwise rotation vector field on $X$, and $H$ is multiplication by $\pi^{-1}H$. We note that $\mathbf{D}$ decompose $L^2(X)$ into eigenspaces $\oplus_{k \in \Z} L^2(X)_k$ with eigenvalue $k \in \Z$. 
The symbol of $T$ is a function on $\Sigma \cong X \times \R_+$, given by 
\[ \sigma(T) (x, r)=  (\sigma(H) \sigma(\mathbf{D})|_\Sigma)(x,r) = H(x) r, \quad \forall (x,r) \in \Sigma. \]
\bd[\cite{RS}, Definition 5.1]
Let $\h U(t)$ denote the one-parameter subgroup of unitary operators on $L^2(X)$, given by 
\be \h U(t) = \hPi\, e^{i t \hPi (\mathbf{D} H) \hPi} \, \hPi, \ee
and let $\h U_k(t)$ \eqref{UkDEF} denote the Fourier component acting on $L^2(X)_k$:
\be \h U_k(t)  =  \hPi_k \, e^{i t \hPi (k H) \hPi} \, \hPi_k. \label{UkDEFb} \ee
We use $U_k(t)$ to denote the corresponding operator on $H^0(M, L^k)$. 
\ed

\bp[\cite{RS}, Proposition 5.2]
$\h U(t)$ is a group of Toeplitz Fourier integral operators on $L^2(X)$, whose underlying canonical relation is the graph of the time $t$ Hamiltonian flow of $r H$ on the sympletic cone $\Sigma$ of the contact manifold $(X,\alpha)$. 
\ep

We warn the reader that $\hPi\, e^{i t \hPi (\mathbf{D} H) \hPi} \, \hPi$ in general is not equal to $\hPi\, e^{i t  \mathbf{D} H} \, \hPi$, since $\mathbf{D} H$ and $\hPi$ may not commute. However, thanks to Lemma \ref{lm:niceQ}, one can always find a $Q$ replacing $\mathbf{D} H$, such that $[Q, \hPi]=0$ and $\hPi (\mathbf{D} H) \hPi = \hPi Q \hPi = Q \hPi = \hPi Q$.  Thus 
\[ U(t_1) \circ U(t_2) = \hPi\, e^{i t_1 \hPi Q \hPi} \, \hPi \hPi\, e^{i t_2 \hPi Q \hPi} \, \hPi = \hPi\, e^{i (t_1+t_2)   Q  } \hPi = U(t_1 + t_2). \]
indeed forms a group. 

\brem
In the introduction, we defined a semi-classical (i.e., depending on $k$) Toeplitz operator $H_k$ acting downstairs on $L^2(M, L^k)$, or equivalently an operator $\h H_k$ acting on $L^2(X)$ by
\[ \h H_k: = \hPi_k \circ [\frac{i}{k} \xi_H^h +  H] \circ \hPi_k .\]
The collections $\{\h H_k\}$ assemble into a homogeneous degree one Toeplitz operator $\h H$ acting on $L^2(X)$:
\[ \h H = \oplus_{k \geq 0}    k \h H_k = \hPi \circ [i \xi_H^h + (-i \pa_\theta) H] \circ \hPi: L^2(X) \to \hcal(X). \]
Compared  with $T=\hPi \circ  [(-i \pa_\theta) H] \circ \hPi$, we claim that they have the same principal symbol on $\Sigma$, indeed
\[ \sigma(\h H) - \sigma(T) = \sigma(i \xi_H^h)|_\Sigma = \la p, -\xi_H^h(x) \ra|_{p = r \alpha(x)} = 0, \]
since $\la \alpha , \xi_H^h \ra = 0$. 
%If we let \[ Q = i \xi_H^h + (-i \pa_\theta) H,\] then
%$ \sigma(Q) = \la p, -\xi_H^h + H \pa_\theta \ra $
%and the Hamiltonian vector field $\xi_{\sigma(Q)}$ on $T^*X$ generated by function $\sigma(Q)$ is induced from the vector field $-\xi_H^h + H \pa_\theta$ on $X$. And since $-\xi_H^h + H \pa_\theta$ preserves $\alpha$, the Hamiltonian vector field $\xi_{\sigma(Q)}$ preserves the symplectic cone $\Sigma$. We warn the reader that, the flow $\xi_{\sigma(Q)}$ on $T^*X$ preserving the symplectic subcone $\Sigma \subset T^*X$ does not imply that  $[Q, \hPi]=0$,even though the converse is always true (see. e.g. proof of Proposition 5.2 in \cite{RS}).  Indeed $[Q, \hPi]=0$ would require $\hPi Q \hPi = Q \hPi$, ie. $Q$ takes $\hcal(X)$ to $\hcal(X)$, which is only possible if the Hamiltonian vector field $\xi_H$ on $(M, \omega)$ preserves the complex structure $J$ as well. 
\erem

When we say a one-parameter family of unitary operators $U(t)$ on $\hcal(X)$ {\em quantizes the Hamiltonian flow} $\exp (t \xi_H)$ on $(M, \omega)$, we mean the real points of  canonical relation of complex FIO $U(t)$ in $T^*X \times T^*X$ is the graph $\{(x, \Psi_t(x)) \mid x \in \Sigma\}$ of the Hamiltonian flow $\Psi_t$ generated by $r H$ on the symplectic cone $(\Sigma, \omega_\Sigma)$. The quantization  is not unique, indeed, if $A$ is a pseudodifferential operator of degree zero on $X$, and $V = e^{i A}$ is unitary pseudodifferential operator, then $V^* U(t) V$ is another quantization with the same principal symbol.

\begin{prop} \label{SC}
There exists a semi-classical symbol  $\sigma_{k}(t)$ so that the unitary group \eqref{UkDEFb}  has the form
\be  \label{TREP}  \hat{U}_k(t)   = \hat{\Pi}_{k}  (\hat{g}^{-t})^* \sigma_{k}(t) \hat{\Pi}_{k}  \ee
modulo smooth kernels of order $k^{-\infty}$.
\end{prop}

\begin{proof}  It follows from the theory of Toeplitz Fourier integral
operators and Fourier integral operators with complex phase \cite{BG81,MS}  that
\eqref{UkDEFb} is a unitary group of semi-classical Fourier integral operators
with complex phase associated to the graph of the lifted Hamiltonian flow
of $H$ to $X$.  On the other hand, the operator on the right hand side
of \eqref{TREP} is manifestly such a Fourier integral operator. To prove
that they are equal modulo smoothing operators of order $k^{-\infty}$ it suffices to construct $\sigma_{k}(t)$ so that they have the same complete symbol.

In \cite{Z97} the principal symbol $\sigma^{prin}_{k}(t)$ of \eqref{UkDEFb} was calculated, and by using
this as the principal symbol of \eqref{TREP} one has that
$$\h U_k(t)  - \hPi_k   (\hat{g}^{-t})^*  \sigma^{prin}_{k}(t)  \hPi_k \in I^{-1}_{sc}(\R \times M \times M, \ccal \circ \hat{\Gamma}
\circ \ccal). $$
By the same method as in the homogeneous case, one  may then construct the Toeplitz symbol  $\sigma_{k}(t)$ to all order so that the complete symbols of $\h U_k(t)$ and
\eqref{TREP} agree to all orders in $k$.
\end{proof}

It follows from the above proposition and the   Boutet de Monvel--Sj\"ostand parametrix  construction that
$\h U_k(t, x, x)$ admits an oscillatory integral representation of the form,
\bee
&& \hat{U}_k (t, x, x)  \notag \\
&\simeq&  \kk^{2m}\int_X \int_0^{\infty}    \int_0^{\infty} \int_{S^1}  \int_{S^1} 
e^{  \sigma_1 \h \psi(r_{\theta_1} x,  \h g^t y) + \sigma_2 \h \psi(r_{\theta_2} y, x) - i k \theta_1 -  i k \theta_2}  
  A_k(t, y, \theta_i, \sigma_i)   d \theta_1 d \theta_2 d \sigma_1  d \sigma_2 d y  \label{hU}  \\
& \simeq & \kk^{2m} \int_X   
e^{  k \h \psi( x,  \h g^t y) + k\h \psi( y, x) }  A_k(t, y)    d y,  \label{hU2}  
\eee
where $A_k$ is a semi-classical symbol, and the integral for $\theta_1, \theta_2$ is to extract the Fourier component, and the integral for $y$ comes from operator composition. The asymptotic symbol $\simeq$ means that the difference of the two sides is rapidly decaying in $k$. 

\section{Short time propagator $\h U_k(t, x,x)$ and Proof of Theorem 3}
In this section, we give an explicit expression for the short time propagator $\h U_k(t)$, by applying stationary phase method to the integral \eqref{hU}. As a warm up, we  illustrate
the quantization of Hamiltonian flows in the model case of quantizations of linear  Hamiltonians function (constant flow) on Bargmann-Fock space. In effect, this is the construction of the Bargmann-Fock representation of the Heisenberg
group in terms of Toeplitz quantization. 
Quantizations of linear Hamiltonians on $\C^m$ are `phase space translations', and belong to the representation theory
of the Heisenberg group. Their relevance to our problem is that Hamilton flows on general \Kahler manifolds $(M, J, \omega)$ can be
approximated by the linear models at a tangent space $(T_z M, J_z, \omega_z)$. Quadratic Hamiltonians and 
the metaplectic representation are constructed by the Toeplitz method in \cite{Dau80}.

\subsection{Propagator with linear Hamiltonian function \label{BFmodel3}}
Let $M = \C^m$, $\omega = i \sum_j dz_j \wedge d \bar z_j$ be the Bargmann-Fock model, and $H = \Re(\alpha \cdot \bar z)$ (see \S \ref{BFmodel1} and \S \ref{BFmodel2}). 

\bp
The  kernel for the propagator $\hat{U}_k(t) = \hPi_k e^{i k t \hat{H}_k} \hPi_k$, is then given  by 
\be \h U_k (t, \h z, \h w) = \kk^{m} e^{i k (\theta_z+ \half \Re(\alpha \bar z)t - \theta_w + \Im ( (z - \frac{\alpha t}{2i}) \bar w)) - k \half |z - \frac{\alpha t}{2 i} - w|^2}. \ee
In particular, if $\h z = \h w$, we have 
\be \h U_k (t, \h z, \h z) =  \kk^{m} e^{i k H(z) t} e^{- k t^2 \frac{\|\xi_H(z)\|^2}{4}}. \ee
\ep
\bpf
We may start from the integral expression
\be \hat{U}_k(t; \hat z, \h w) =  \kk^{2m}\int_{\hat{y} \in X} e^{k [\h \psi(\hat{z} , \h g^t(\hat{y})) +  \h \psi( \hat{y}, \hat{w} )]}  d \Vol_X(\hat{y}) 
\ee
where the function   $\h \psi(\h x_1, \h x_2)$ is given by
\[ \h \psi(\h x_1, \h x_2) = i (\theta_1 - \theta_2) + z \bar w - |z|^2/2 - |w|^2/2, \quad \h x_1 = (z, \theta_1), \h x_2 = (w, \theta_2)\]
and $\h g^t(\h y)$ is given by
\[ \h g^t(\h y) = (y + \frac{\alpha t}{2 i}, \theta_y - \half H(y) t). \]
Then after a Gaussian integral of $dy$ we get the desired answer. The diagonal case is straightforward, except to note that $\|\xi_H(z)\|^2 = \omega(\xi_H(z), J \xi_H(z)) = \half |\alpha|^2$. 
\epf

\brem
Since the Hamiltonian flow for a linear Hamiltonian is just translation on $\C^m$, which is compatible with the complex structure on $\C^m$,  we have $\h H_k$ commute with $\h \Pi_k$, hence $\h U_k(t) = \h \Pi_k (\h g^{-t})^* \h \Pi_k =   (\h g^{-t})^*  \h \Pi_k$. Thus 
\[ \h U_k(t; \h z, \h w) = \h \Pi_k(\h g^{-t} \h z, \h w)  = \kk^{m} e^{k \h \psi(\h g^{-t} \h z, \h w)}, \]
which gives the desired result. 
\erem
\ss{Short time propagator $\h U_k(t, \h z, \h z )$ for $t \sim 1/\sqrt{k}$. }
 Let $(M, \omega)$ be a \Kahler manifold polarized by ample line bundle $(L,h)$, such that $\omega =  c_1(L)$. Let $X \to M$ be the dual circle bundle over $M$. We have the following result about the kernel $\h U(t)$ on the diagonal, generalizing the special case of Bargmann-Fock model. 
\bp \label{p:sqrtkU}
If $z_0 \in M$ such that $dH(z_0) \neq 0$, then for any $\tau \in \R$,
\[ \h U_k(\tau/\sqrt{k},\h z_0, \h z_0) = \kk^{m} e^{i \tau \sqrt{k} H(z_0)} e^{-\tau^2 \frac{\|\xi_H(z_0)\|^2}{4}} (1 + O(|\tau|^3k^{-1/2})), \]
where the constant in the error term is uniform as $\tau$ varies over compact subset of $\R$. 
\ep
\bpf
First, we pick \Kahler normal coordinate $z_1, \cdots, z_m$ centered at point $z_0$. That is, for a small enough $\epsilon>0$, we may find coordinate in $U=B(z_0, \epsilon)$ such that the coordinate of $z_0$ is $0 \in \C^m$, and 
\[ \omega(z) = i \sum_{j=1}^m d z_j \wedge d \bar z_j + O(|z|^2). \]
We may also choose a local reference frame $e_L$ of the line bundle in a neighborhood of $z_0$, such that the induced \Kahler potential $\varphi$ takes the form
\[ \varphi(z) = |z|^2 + O(|z|^3). \]
And the almost analytic continuation is
\[ \psi(z,w) = z \cdot \bar w + O(|z|^3 + |w|^3). \]
%These are called K-coordinates and K-frame in Lu-Shiffman \cite{LuSh15}. 
The Hamiltonian assume the form of 
\[ H(z) = H(0) + H_1(z) + O(|z|^2)\]
where $H_1(z)$ is a real valued linear function on $\C^m \cong \R^{2m}$. Locally the circle bundle $X$ over $M$ can be trivialized with fiber coordinate $\theta$.  

Now we are ready to evaluate the integral in \eqref{hU2}
\[ \h U_k(\tau/\sqrt{k}, \h z_0, \h z_0) = \kk^{2m}  \int_{S^1} \int_{M} e^{k[\h \psi(0, \h g^{t} \h w) + \h \psi (\h w, 0)]} A_k dw \frac{d\theta_w}{2\pi} + O(k^{-\infty}). \] 
If $|w| > \epsilon$ or $|g^t w| > \epsilon$, the integrand is fast decaying in $k$, due to the off-diagonal decay of $\Pi_k(z_1, z_2)$. Hence we may restrict the integral to $|w| < \epsilon$ and $|g^t w| < \epsilon$. 

The phase function in the integral \eqref{hU2}, when we plug in 
\[ w = u/\sqrt{k}, \quad t = \tau/\sqrt{k}, \]
 and write $\h w = (w, \theta_w)$ and $\h g^t \h w = (w(t), \theta_w(t))$, then if $|w|<\epsilon, |g^t w| < \epsilon$, we have 
\bea \h \psi(0, \h g^{t} \h w) + \h \psi (\h w, 0) &= & i (\theta_w(0) - \theta_w(t))  - |w(0)|^2/2 - |w(t)|^2/2 + O(|w|^3 + |w(t)|^3) \\
&=& k^{-1/2} [i H(0) \tau] + k^{-1}[i \half H_1(u) \tau - |u|^2/2 - |u + \xi_{H_1} \tau|^2/2] + O(k^{-3/2} (|u|^3+|\tau|^3)). 
\eea
Thus we have
\bea 
&& \h U_k(t,  \h z_0, \h z_0) \\
&=&\kk^{2m}  \int_{S^1} \int_{B(z_0, \epsilon) \cap g^{-t}(B(z_0, \epsilon))} e^{k[\h \psi(0, \h g^{t} \h w) + \h \psi (\h w, 0)]} A_k dw \frac{d\theta_w}{2\pi} + O(k^{-\infty})\\
&=& \kk^{2m}  \int_{S^1} \int_{B(z_0, \epsilon) \cap g^{-t}(B(z_0, \epsilon))} e^{k[i (\theta_w(0) - \theta_w(t))  - |w(0)|^2/2 - |w(t)|^2/2 + O(|w|^3 + |w(t)|^3)]} A_k  dw d \theta_w + O(k^{-\infty})  \\
&=& \kk^{2m}  \int_{S^1} \int_{B(z_0, \epsilon) \cap g^{-t}(B(z_0, \epsilon))} e^{k[i (\theta_w(0) - \theta_w(t))  - |w(0)|^2/2 - |w(t)|^2/2]}    dw d \theta_w (1 + O(k^{-1/2}) )\\
&=& \kk^{2m}  \int_{|u|<\epsilon k^{1/2}} e^{k^{1/2} [i H(0) \tau] +[i \half H_1(u) \tau - |u|^2/2 - |u + \xi_{H_1} \tau|^2/2]}    k^{-m} d\Vol_{\C^m}(u)  (1 + O(|\tau|^3 k^{-1/2}) )\\
&=& \kk^{2m} e^{k^{1/2} [i H(0) \tau]} \int_{u \in \C^m} e^{[i \half H_1(u) \tau - |u|^2/2 - |u + \xi_{H_1} \tau|^2/2]}   k^{-m} d\Vol_{\C^m}(u)  (1 + O(|\tau|^3k^{-1/2}) )\\
&=&\kk^{m} e^{i \tau \sqrt{k} H(z_0)} e^{-\tau^2 \frac{\|\xi_H(z_0)\|^2}{4}} (1 + O(|\tau|^3k^{-1/2})),
\eea
where the calculation in the last step is the exactly the same as in the Bargmann-Fock model. 
\epf

\ss{Short time propagator $U_k(t, z, z )$ for non-periodic trajectories.}
For $z \in M$ such that $dH(z) \neq 0$, let $T(z) \in (0, \infty]$ be the period of the flow trajectory starting at $z$, and if the trajectory is not periodic we let $T(z)=\infty$. 

\bl \label{lm:decayU}
The function $T(z)$ is uniformly bounded from below over compact subset in the complement of the fixed point set of $\xi_H$. 
\el
\bpf
Let $K \subset M$ be a compact subset, such that if $z \in K$ then $dH(z) \neq 0$. Then for small enough $\epsilon>0$, the graph of the flow $\xi_H$
\[\Phi: K \times (-\epsilon, \epsilon) \to M \times M \quad (z, t) \mapsto (z, \exp(t \xi_H)(z)) \]
is injective by the implicit function theorem. Thus $T(z) > \epsilon$ for all $z \in K$. 
\epf

\bl
Fix any $T>0$. For any $z,w \in M$ and $t \in \R$ with $|t|<T$, there exists constant $C, \beta > 0$ depending on $(T, H,  M, \omega, J)$, such that for any $\h z, \h w \in X$ lift of $z,w$, we have
\[ |\h U_k(t, \h z, \h w)| < C k^m e^{-\beta \sqrt{k} d(z, g^t w)}. \]
where $d$ is the Riemannian metric of $M$. In addition, for any $n>0$, we have
\[ |\pa_t^n \h U_k(t, \h z, \h w)| < C_n k^{m+n} e^{-\beta \sqrt{k} d(z, g^t w)}. \]
for some constants $C_n$. 
\el
\bpf
Modulo a smooth remainder term, we have
\[ \h U_k(t, \h z, \h w) = \int_{\h u \in X} \h \Pi_k(\h z,\h g^t \h u) \h \Pi_k(\h u, \h w) A_k(t, \h z, \h u, \h w) d \h u. \]
From the off diagonal decay estimate of the \szego kernel  (Theorem \ref{AGMON1}), there exists $C_1, \beta_1$, such that
\[ |\h \Pi_k(\h u, \h v)| < C_1 k^m e^{-\beta_1 \sqrt{k} d(u, v)}. \]
Thus there exists constant $C_2>0$ such that,
\[ |\h U_k(t, \h z, \h w)| < C_2 \sup_{\h u \in X} |\h \Pi_k(\h z,\h g^t \h u)| \cdot |\h \Pi_k(\h u, \h w)| < C_1C_2 k^m \sup_{ u \in M} e^{-\beta_1 \sqrt{k} (d(z, g^t u) + d(u, w))}. \]
To finish the proof, suffice to bound $d(z, g^t u) + d(u, w)$ from below by $d( z, g^t w)$.
Since $g^t$ is a diffeomorphism on $M$, there exists $1 > \epsilon>0$, such that
\[ d(g^t u, g^t v) > \epsilon d(u, v), \quad \text{for all } u, v \in M,  |t| < T,\]
then
\[  d(z, g^t u) + d(u, w) > \epsilon d(z, g^t u) + \epsilon d(g^t u, g^t w) \geq \epsilon d (z, g^t w). \]
To prove the claim for $t$-derivative, we note that
\[ \pa_t^n \h U_k(t) = (i k \h H_k)^n \h U_k(t) \]
since $\h H_k$ is an operator with bounded spectrum on $\hcal_k(X)$, we get the desired estimate. 
\epf

\ss{Proof of Theorem \ref{ELLSMOOTH}.}
\bpf
By the Fourier inversion formula,
\[ \int_\R f(x) d \mu^{z,1, \alpha}_k(x) = \sum_{j} f( k (\mu_{k,j} - H(z)) + \sqrt{k} \alpha) \hPi_{k,j}(\h z, \h z) = \int e^{- i kt  (H(z) - \alpha/\sqrt{k})} \h f(t) \h U_k(t, \h z, \h z) \frac{dt}{2\pi}. \] To complete the proof of  the Theorem, it suffices to prove

\begin{lem} \label{int-t}
Let $K$ be a compact subset of $M$ that does not contain any fixed points of $\xi_H$, and let $\epsilon>0$ be small enough such that $T(z) > \epsilon$ for all $z \in K$. Then for any $z \in K$ and any non-negative Schwarz function $f\in \scal(\R)$ such that $\int f(x) dx = 1$, i.e. $\h f(0)=1$, and $\h f(t)$ is supported in $(-\epsilon, +\epsilon)$, and for any $\alpha \in \R$, we have 
\be \label{e:epst} \int_\R \h f(t) U_k(t, z, z) e^{-i t k H(z) + i t \sqrt{k} \alpha} \frac{dt}{2\pi} = \kk^{m-1/2} e^{- \frac{\alpha^2}{\|\xi_H(z)\|^2} }\frac{\sqrt{2}}{2\pi \|\xi_H(z)\|}(1+ O(k^{-1/2})). \ee
\end{lem}
Fix any $0 < \delta < 1/2$, we claim that the contribution to the integral is $O(k^{-\infty})$ if $\epsilon>|t| > k^{-1/2+\delta}$.  From Lemma \ref{lm:decayU}, 
\[  U_k(t, z, z) = \h U_k(t,\h z, \h z) \leq C k^m e^{- \sqrt{k} \beta d(z, g^t z)} \leq C k^m e^{-\beta' k^{\delta}} =O(k^{-\infty}), \quad \forall |t| \in  [k^{-1/2+\delta}, \epsilon]. \]
Now, we may approximate the integral on the left side of \eqref{e:epst}, and rescale the integration variable by $t = \tau / \sqrt{k}$, 
\[ \int \h f(t) U_k(t, z, z) e^{-i t k H(z) + i t \sqrt{k} \alpha}  \frac{dt}{2\pi} = k^{-1/2} \int_{|\tau|  < k^{\delta}} \h f(\tau k^{-1/2}) U_k(\tau k^{-1/2}, z, z) e^{-i \tau \sqrt{k} H(z) + i \tau \alpha} \frac{d \tau}{2\pi}. \]
Taylor expanding $\h f(\tau k^{-1/2})  = \h f(0)(1+O(|\tau| k^{-1/2})) =1+O(|\tau| k^{-1/2})$ using $\h f(0) = 1$,  and using the expansion of $U_k(\tau k^{-1/2}, z, z)$ in Proposition \ref{p:sqrtkU}, we have
\bea \int \h f(t) U_k(t, z, z) e^{-i t kH(z) + i t \alpha \sqrt{k}} \frac{dt}{2\pi}  &=& k^{-1/2} \int_{|\tau|  < k^{\delta}} \kk^{m}  e^{-\tau^2 \frac{\|\xi_H(z)\|^2}{4} + i \tau \alpha} (1 + O(|\tau|^3 k^{-1/2})) \frac{d\tau}{2\pi} \\
&=& \kk^{m-1/2}e^{- \frac{\alpha^2}{\|\xi_H(z)\|^2} }\frac{\sqrt{2}}{2\pi \|\xi_H(z)\|}(1+ O( k^{-1/2})). 
\eea
This completes the proof of Lemma \ref{int-t} and of Theorem \ref{ELLSMOOTH} . 
\epf
%
%
%
%hence Theorem 3 is equivalent to Proposition \ref{int-t}. 

\section{\label{muhalfsect}Smooth density of states: Proof of  Theorem \ref{RSCOR}-(1)}

We start by proving \eqref{thm-2-1}. 
As above, let  $\hat{g}^t: X \to X$ denote the lifted Hamiltonian flow
and let $F^t$ denote the gradient flow.  For any $z \in M$, we choose a lift $\h z$ in $X$, since $\h \Pi_k(\h z, \h z)$ and $\h U_k(t, \h z, \h z)$ does not depend on the choice of the lift, we will use $\h z$ and $z$ interchangeably in these cases. 
 
\bpf [Proof of  \eqref{thm-2-1} ]

Let $z_k = F^{ \frac{\beta}{\sqrt{k}}} z$. By \eqref{fofHFT} we have,
\[
I := \sum_{j = 1}^{d_k} f( \sqrt{k} (\mu_{k,j} - E)) \Pi_{k,{\mu_{k,j}}}(z_k, z_k)  =   \int_{\R} \hat{f}( t) e^{- i E \sqrt{k} t} {U}_{k}( t/\sqrt{k},  z_k, z_k) \frac{dt}{2\pi}
\]
Using Proposition \ref{p:sqrtkU}, we have expression for $U_k$, 
\[ I = \kk^m \int_{\R} \hat{f}( t) e^{- i E \sqrt{k} t}  k^m  e^{ i t \sqrt{k} H(z_k)} e^{- \frac{1}{4}  |t \xi_H(z_k)|^2} [1 +O(|t|^3k^{-\half})]  \frac{dt}{2\pi}  \]
In the exponent, using $E = H(z)$, we get
\bea 
- i E \sqrt{k} t + i t \sqrt{k} H(z_k) &=& i t \sqrt{k} (H(z_k) - H(z)) \\
&=& i t \sqrt{k} [g( \nabla H(z), \frac{\beta}{\sqrt{k}} \nabla H(z)) + O((\beta/\sqrt{k})^2) \\
&=& i t \beta \|\nabla H(z)\|^2 + O(|t| k^{-\half}) 
\eea
Furtheremore, $- \frac{1}{4}  |t \xi_H(z_k)|^2 = - \frac{1}{4}  |t \xi_H(z)|^2+ O(k^{-\half} |t|^2) $. 

Hence
\bea 
I &=&    \kk^m  \int_{\R} \hat{f}( t)  e^{ i t \beta \|\nabla H(z)\|^2 - \frac{1}{4}  |t \nabla H(z)|^2} \frac{dt}{2\pi} [1 +O(k^{-\half})] \\
&=& \kk^m  \int_{t\in \R} \int_{x \in \R} f(x) e^{-i x t}  e^{ i t \beta \|\nabla H(z)\|^2 - \frac{1}{4}  |t \nabla H(z)|^2} \frac{dx dt}{2\pi} [1 +O(k^{-\half})] \\
&=& \kk^m  \int_{x \in \R} f(x) e^{ -\left(\frac{x}{\|\nabla H\|} - \beta \| \nabla H(z) \|\right)^2} \frac{dx}{\sqrt{\pi} \| \nabla H(z) \|}  [1 +O(k^{-\half})]
\eea

 To obtain  the complete asymptotic expansion stated in Theorem \ref{RSCOR}(1), it is only necessary to  Taylor expand  $F^{ \frac{\beta}{\sqrt{k}}} z$ and  use the expansion of Proposition \ref{p:sqrtkU} for  $\h U_k(t/\sqrt{k})$  to higher order in $k^{-\half}$.  
\epf

\section{Tauberian Argument: Proof of Theorem \ref{RSCOR}-(2) and \ref{RSCOR}-(3).}

We denote by $W$ a positive function in $\scal(\R)$ with $\int W(x) dx = 1$ and $\supp \h W \subset (-\epsilon, \epsilon)$, where $\epsilon$ is small enough such that there is no closed orbit of flow $\xi_H$ with passing through $z$ with time less than $\epsilon$. Let $W_h(x) = h^{-1} W(x/h)$, $h>0$.  

For clarity of notation, we let $F_k(x) = \Pi_k(z)^{-1} \mu^{z,1/2}_k (x)$.  
\bp \label{FW}
There exists  a large enough $k_0$, such that for any $k > k_0$, and for all $x \in \R$, we have
\[ F_k * W_{k^{-1/2}} (x_0) = \int^{x_0}_{-\infty} e^{-\frac{x^2}{|\xi_H|^2}}\frac{dx}{\sqrt{\pi} |\xi_H|} + O(k^{-1/2}).   \]
\ep
\bpf
Denote the $F_k * W_{k^{-1/2}} (x_0) $   by $I_1$, and let $h = k^{-1/2}$, then
\bea
I_1 & = &\int^{x_0}_{-\infty} \int_{y \in \R} \Pi_k(z)^{-1} \sum_{j=1}^{N_k} \| s_{k,j}(z) \|^2  \delta_{ \sqrt{k} (\mu_{k,j}-H(z))}(y) W_h(x-y) dy dx \\
&=& \int^{x_0}_{-\infty}\Pi_k(z)^{-1}\sum_{j=1}^{N_k} \| s_{k,j}(z) \|^2 W_h(x - \sqrt{k} (\mu_{k,j}-H(z))) dx\\
&=& \int^{x_0}_{-\infty}\Pi_k(z)^{-1}\sum_{j=1}^{N_k} \| s_{k,j}(z) \|^2 \int_\R \h W_h (\tau) e^{i \tau [x - \sqrt{k} (\mu_{k,j}-H(z))]} \frac{d\tau dx}{2\pi}  \\
&=&\Pi_k(z)^{-1}\int^{x_0}_{-\infty}  \int_\R  e^{i \tau x} \h W (\tau/\sqrt{k}) e^{i \tau \sqrt{k} H(z)} U_k(-\tau/\sqrt{k}, z, z)  \frac{d\tau dx}{2\pi}
\eea

Since $\supp \h W \subset (-\epsilon, \epsilon)$, the integral of $t$ is limited to $| t/\sqrt{k}| < \epsilon$. We first show that one can further cut-off the $d\tau$ integral to $|\tau | < k^{\delta}$, for any $0<\delta \ll 1/2$. Let $\chi(x) \in C^\infty_c(\R)$ which is constant $1$ in a neighborhood of $x=0$. Then we claim the following 
\be \label{cut-tau}
\Pi_k(z)^{-1}\int^{x_0}_{-\infty}  \int_\R    (1-\chi(\tau / k^\delta)) \h W (\tau/\sqrt{k}) e^{i \tau \sqrt{k} H(z) + i \tau x} U_k(-\tau/\sqrt{k}, z, z)  \frac{d\tau dx}{2\pi} = O(k^{-\infty})
\ee
Indeed, we may define
\[ G_k(\tau):=(1-\chi(\tau / k^\delta)) \h W (\tau/\sqrt{k}) e^{i \tau \sqrt{k} H(z)} U_k(-\tau/\sqrt{k}, z, z),\]  and let $\h G_k(x)$ be its Fourier transformation. Then \eqref{cut-tau} can be written as 
\[ \int_{-\infty}^{x_0} \int_\R e^{i \tau x} G_k(\tau) d \tau dx /2\pi = \int_{-\infty}^{x_0} \h G_k(x) dx. \]
We note that $G_k(\tau)$ is a smooth and compactly supported function in $\tau$, hence its Fourier transform $\h G_k(x)$ is also a Schwarz function in $x$. Since when $ k^{-1/2+\delta} < |\tau|/\sqrt{k} <\epsilon$, we have $U_k(-\tau/\sqrt{k}, z, z)$ and all its derivatives in $\tau$ are bounded by $k^\gamma e^{-\beta k^{\delta}}$ for some $\beta, \gamma>0$, hence all the Schwarz function semi-norms of $G_k(\tau)$ and $\h G_k(x)$ are $O(k^{-\infty})$. Thus proving \eqref{cut-tau}. 

With the above claim, we can write
\bea
I_1 &=&\Pi_k(z)^{-1} \int^{x_0}_{-\infty}  \int_\R    \chi(\tau/k^\delta) \h W (\tau/\sqrt{k}) e^{i\tau x+ i \tau \sqrt{k} H(z)} U_k(-\tau/\sqrt{k}, z, z)  \frac{d\tau dx}{2\pi} + O(k^{-\infty})\\
&=& \int^{x_0}_{-\infty}  \int_\R    \chi(\tau/k^\delta) \h W (\tau/\sqrt{k}) e^{i\tau x}  e^{-\tau^2 \frac{\|\xi_H(z)\|^2}{4}} (1 + R_k(\tau)) \frac{d\tau dx}{2\pi} + O(k^{-\infty}) \\
\eea
where $e^{-\tau^2 \frac{\|\xi_H(z)\|^2}{4}} R_k(\tau)$ is a Schwarz function in $\tau$ with all the Schwarz norm bounded by $k^{-1/2}$. With the Gaussian suppression factor $e^{-\tau^2 \frac{\|\xi_H(z)\|^2}{4}}$, we may replace the factor $\chi(\tau/k^\delta)$ by $1$, while adding an error term of $O(k^{-\infty})$. We may also write 
\[ \h W (\tau/\sqrt{k}) = \h W(0) + O(|\tau|/\sqrt{k} )= 1 + O(|\tau|/\sqrt{k}),\]
and absorb the error term into $R_k(\tau)$ as well. Thus we have
\bea
I_1 &=&  \int^{x_0}_{-\infty}  \int_\R  e^{i\tau x} e^{-\tau^2 \frac{\|\xi_H(z)\|^2}{4}} (1 + R_k(\tau)) \frac{d\tau dx}{2\pi} + O(k^{-\infty}) \\
&=&  \int^{x_0}_{-\infty} e^{-\frac{x^2}{|\xi_H|^2}}\frac{dx}{\sqrt{\pi} |\xi_H|} + O(k^{-1/2}).
\eea
\epf
 
This shows the smoothed measure of $d \mu^{z,1/2}_k$ has the desired cumulative distribution function. To prove the sharp result, we need a Tauberian argument.

\begin{lem}\label{TAUBERLEM} Let $\sigma_h = d F_h$ be a family
of finite measures,  where  $F_h: \R \to \R$ is a family of non-degreasing functions with $h \in [0,1)$ and satisfying
\begin{enumerate}

\item $\rm{supp} \sigma_h (x)  \subset [- h^{-1}, h^{-1} ]$;\bigskip

\item $F_h(x) = \sigma_h[-\infty, x]$; \bigskip

\item There exists a non-negative integer $n$ so that $F_h(x) \leq h^{-n}$ uniformly in $x$ as $h \to 0$.\bigskip

\item $\frac{d}{dx} F_h * W_h (x) = O(h^{-n})$ uniformly in $x$; 

\bigskip

\end{enumerate}

Then
$$F_h(x) = F_h * W_h (x) + O(h^{-n + 1}), \;\; h \to 0. $$

\end{lem}

The Lemma is almost the same as in \cite[Theorem V-13, p. 266]{R87}, except that the latter Tauberian lemma assumes that $\rm{supp} d \F_h$
is a fixed interval $[\tau_0, \tau_1]$ whereas our $d \mu_k^{z, \half}$
have supports in $C [-\sqrt{k}, \sqrt{k}]$. It turns out that the proof 
of  \cite[Theorem V-13, p. 266]{R87} extends to this situation with no change
in the proof.  For the sake of completeness, we review the proof in the Appendix \ref{TAUBERAPP}  to ensure that the extension is correct. \footnote{We thank D. Robert for corroborating that
the compact support condition is unnecessary.}

\bp
If we let $h=k^{-1/2}$  and $n=0$, and $F_h = F_{h_k} = \kk^{-m} \mu_k^{z,1/2}(-\infty, x)$. Then we have\\
(1) $\sup_{x} F_h(x) < C$ for some positive constant $C$. \\
(2) $\frac{d}{dx} F_h * W_h (x) = O(1)$ uniformly in $x$; \\
\ep
\bpf
(1) Since $F_h(x)$ is non-decreasing, and $\lim_{x \to \infty} F_h(x) = \Pi_k(z)^{-1} \Pi_k(z) = 1$, hence $F_h(x)$ is uniformly bounded by $1$.

(2) By a similar argument used in Proposition \ref{FW}, we have  
\[ \frac{d}{dx} F_h * W_h (x) = e^{-\frac{x^2}{|\xi_H|^2}}\frac{1}{\sqrt{\pi} |\xi_H|(z)} + O(k^{-1/2}), \]
hence is uniformly bounded in $x$. 
\epf

In particular, the condition in Lemma \ref{TAUBERLEM} is satisfied, and we have
\[ F_h (x_0) = F_h *W_h(x_0) + O(k^{-1/2}) =  \int^{x_0}_{-\infty} e^{-\frac{x^2}{|\xi_H|^2}}\frac{dx}{\sqrt{\pi} |\xi_H|} + O(k^{-1/2}). \]

\subsection{Proof of  Theorem 2-(3)}

In the statement of Theorem 2-(2), the sequence of points $(z_k, E_k) =(z, H(z) + \alpha/\sqrt{k})$ approaches $(z,H(z))$ while keeping $z_k$ fixed.  The following proposition would relax the direction of approaching. 
\bp
Let $(L, h, M, \omega)$ and $H, E$ be as Theorem \ref{RSCOR}. Pick any $z \in H^{-1}(E)$. If there is a sequence of $(z_k, E_k) \in M \times \R$ and $\alpha \in \R$, such that $|z_k - z| = O(k^{-1/2})$ and $|E_k - E| = O(k^{-1/2})$, and $|\sqrt{k}(E_k - H(z_k)) - \alpha| = O(k^{-1/2})$, then 
\be \Pi_{k, E_k}(z_k) = \kk^m \Erf\left( \frac{  \sqrt{2} \alpha}{|\nabla H|(z)} \right) + O(k^{m-1/2}). \ee
\ep
 
\bpf
By Theorem 2-(2),  the constant in \eqref{eqmuhalf} is uniform for $z$ for a compact neighborhood $K$ (not containing the critical point of $H$) and $|\alpha| < T$. Thus,  we have for $k$ large enough, $z_k \in K$, and 
\[ \Pi_{k, H(z_k) + \alpha/\sqrt{k}}(z_k) = \kk^m \Erf\left( \frac {\sqrt{2} \alpha}{|\nabla H|(z_k)} \right) + O(k^{m-1/2}). \]
Let $\alpha_k = \sqrt{k} (E_k - H(z_k))$, then by hypothesis $|\alpha - \alpha_k| = O(k^{-1/2})$, then we have
\[  \Pi_{k, E_k}(z_k) = \Pi_{k, H(z_k) + \alpha_k/\sqrt{k} }(z_k) = \kk^m \Erf\left( \frac{\sqrt{2} \alpha_k}{|\nabla H|(z_k)} \right) + O(k^{m-1/2}) .  \]
\epf

Given the above Proposition, we may take the sequence $(z_k, E_k) = (F^{\beta / \sqrt{k}} (z), H(z))$,  and $\alpha = -\beta |\nabla H(z)|^2$, then verify that
$$ \sqrt{k}(E_k - H(z_k)) - \alpha =  -\sqrt{k}(H(z) + \la dH, \nabla H \ra (\beta/\sqrt{k}) + O(k^{-1}) - H(z)) + \beta | \nabla H(z) |^2 = O(k^{-1/2}) $$
Thus
\[\Pi_{k, H(z)}(F^{\beta / \sqrt{k}} (z)) = \kk^m \Erf\left( -\sqrt{2}\beta |\nabla H|(z)  \right) + O(k^{m-1/2}). \]
this proves \eqref{REMEST}. 
\appendix

\section{Off-diagonal decay estimates}

\begin{theo} (See Theorem 2 of \cite{Del} and Proposition 9 of \cite{L})] \label{AGMON1}  Let $M$ be a compact
\kahler manifold, and let $(L, h) \to M$ be a positive Hermitian line
bundle.  Then the exists a constant $\beta=\beta(M,L,h)>0$ such that
$$|\tilde\Pi_N(x, y)|_{\tilde h^N} \leq  CN^{m} e^{-\beta \sqrt{N} {d} (x, y)}.  $$
where ${d}(x,y)$ is the Riemannian distance with respect to
the \kahler metric $\tilde{\omega}$.\end{theo} 

The theorem is stated for strictly pseudo-convex domains in $\C^n$ but applies with no essential change to unit codisc bundles of positive Hermitian line bundles.

\section{\label{TAUBERAPP}Tauberian theory}

In this section, we  review  the semi-classical Tauberian theorem 
of Robert  \cite{R87}.

Let $\theta \in C_0^{\infty}(-1,1) $ satisfy $\theta(0) = 1, $ and $\theta(-x) = \theta(x)$. We may also
assume $\hat{\theta} \geq 0$ and $|\hat{\theta}(x)| \geq r_0$ for $|x| \leq \delta_0$.
Let $$W_h(x) = (2 \pi h)^{-1} \hat{\theta} (- \frac{x}{h}), $$

\begin{theo}\label{TAUBER} Let $\sigma_h: \R \to \R$ be a family of non-degreasing functions with $h \in [0,1)$ and satisfying
\begin{enumerate}
\item $\sigma_h (x) = 0$ for $x \leq x_0$;\bigskip

\item There exists $x_1 > x_0$ so that  $\sigma_h$ is constant on $[x_1, \infty]$;
\bigskip

\item There exists a positive integer $n \geq 1, $ so that  $\sigma_h(x) \leq h^{-n}$ uniformly in $x$ as $h \to 0$.\bigskip

\item $\frac{d}{dx} \sigma_h * W_h (x) = O(h^{-n})$ uniformly in $x$; 

\bigskip

\end{enumerate}

Then
$$\sigma_h(x) = \sigma_h * W_h (x) + O(h^{-n + 1}), \;\; h \to 0. $$

\end{theo}

\begin{proof}

One has
$$\begin{array}{lll} \sigma_h(\tau) - \sigma_h * W_a(\tau) &=& \int_{\R} (\sigma_h(\tau) - \sigma_h(\tau - \mu)) W_a(\mu)
d\mu \\&&\\
& = & \int_{\R} (\sigma_h(\tau) - \sigma_h(\tau - a \nu )) \hat{\theta}(-\nu)d \nu. 
\end{array}. $$

Due to the fact that $\hat{\theta} \in \scal(\R)$ and assumptions (1)-(2), the Theorem reduces to the following estimate on the $a$-scale
increments.

\begin{lem}\label{ROBERTLEM} There exists   $ \Gamma >0$: 
$$|\sigma_h(\tau) - \sigma_h(\tau + h\nu)| \leq \; \Gamma  (|\nu| +1) \;h^{-n +1}, $$
for all $\tau, \nu \in \R$.
\end{lem}

The proof is broken up into three cases. 

\noindent{\bf Case (i)}  $|\nu|\leq \delta_0$, where $|\hat{\theta}| \geq r_0$ on $|x| \leq \delta_0$. 
\bigskip

We have,
$$|\sigma_h(\tau) - \sigma_h(\tau + h\nu)| = \int_{\tau - h |\nu|}^{\tau + h |\nu|} d \sigma_h(\mu)  \leq h \; d \sigma_h * W_h(\tau) =  \int_{\R} \hat{\theta}(\frac{ \tau -\mu}{h}) d\sigma_h(\mu). $$
The  inequality holds because  $|\mu - \tau| \leq h \delta_0$ on
the interval of integration and $|\hat{\theta}| \geq r_0$ on $|x| \leq \delta_0$. .  
The statement of the Lemma then follows from (4).

\bigskip

\noindent{\bf Case (ii)}  $|\nu| = j \delta_0, j \in \Z$ \bigskip

One has 
$$|\sigma_h(\tau) - \sigma_h(\tau + h j \delta_0)| = \sum_{k =1}^j
|\sigma_h(\tau + h k \delta_0) - \sigma_h(\tau + h (k-1) \delta_0)|.$$
Applying case (i) to each term gives
$$|\sigma_h(\tau) - \sigma_h(\tau + h j \delta_0)| \leq C |j|  h^{-n +1}.$$

\noindent{\bf Case (ii)}  $j \delta_0 < |\nu| < (j +1) \delta_0, j \in \Z$.
\bigskip

In this case,

$$|\sigma_h(\tau) - \sigma_h(\tau + h\nu)| \leq |\sigma_h(\tau) - \sigma_h(\tau + j h \delta_0)| +
|\sigma_h(\tau + j h \delta_0) - \sigma_h(\tau)|. $$
By the previous two cases,
$$|\sigma_h(\tau) - \sigma_h(\tau + h\nu)|  \leq C h h^{-n+1}(1 + |j|).$$

It follows that
$$|\sigma_h(\tau) - \sigma_h(\tau + h\nu)|  \leq a h^{-n} C \delta_0^{-1} (\delta_0 + |\nu|). $$

\end{proof}

\end{document}